\documentclass[11pt]{article}
\usepackage[margin=1.2in]{geometry}

\DeclareFontFamily{OT1}{ptm}{}

\usepackage{subcaption}
\usepackage{amsfonts}
\usepackage{amsthm}
\usepackage{latexsym, amssymb, bbm}
\usepackage{xcolor}
\usepackage{graphicx}
\usepackage{amsmath}
\usepackage{enumerate}
\usepackage{comment}
\usepackage{float}
\usepackage{tikz}
\usepackage{todonotes}
\usetikzlibrary{cd}

\usepackage{amsmath,amsthm,amsfonts,amssymb,graphicx,psfrag}
\usepackage{color}
\usepackage{hyperref}
\hypersetup{
	backref,colorlinks=true,linkcolor=blue,citecolor=blue,urlcolor=blue,citebordercolor={0 0 1},urlbordercolor={0 0 1},linkbordercolor={0 0 1},
}

\newtheorem{theorem}{Theorem}[section]
\newtheorem{lemma}[theorem]{Lemma}
\newtheorem{cor}[theorem]{Corollary}
\newtheorem{question}[theorem]{Question}
\newtheorem{prop}[theorem]{Proposition}
\newtheorem{example}[theorem]{Example}
\newtheorem{definition}[theorem]{Definition}

\newtheorem{conjecture}[theorem]{Conjecture}

\theoremstyle{remark}
\newtheorem{remark}[theorem]{Remark}

\newcommand{\Z}{\mathbb{Z}}

\newcommand{\mc}{\mathcal}
\newcommand{\ZZ}{\mathbb{Z}}
\newcommand{\QQ}{\mathbb{Q}}
\newcommand{\CC}{\mathbb{C}}
\newcommand{\RR}{\mathbb{R}}
\newcommand{\PP}{\mathbb{P}}
\newcommand{\tr}{{\rm tr}}

\title{Circular spherical divisors and their contact topology} 
\author{Tian-Jun Li, Cheuk Yu Mak, Jie Min}

\AtEndDocument{\bigskip{\footnotesize
  \textsc{School of Mathematics, University of Minnesota, Minneapolis, MN, US} \par
  \textit{E-mail address}: \texttt{tjli@math.umn.edu} \par
  \addvspace{\medskipamount}
  \textsc{School of Mathematics, University of Edinburgh, Edinburgh, UK} \par
  \textit{E-mail address}: \texttt{cheukyu.mak@ed.ac.uk} \par
  \addvspace{\medskipamount}
  \textsc{Max Planck Institute for Mathematics, Bonn, Germany} \par
  \textit{E-mail address}: \texttt{minxx127@umn.edu} \par
}}

\date{\today}
 
\begin{document}
\maketitle
\begin{abstract}
	This paper investigates the symplectic and contact topology associated to circular spherical divisors. We classify, up to toric equivalence, all concave circular spherical divisors $ D $ that can be embedded symplectically into a closed symplectic 4-manifold and show they are all realized as symplectic log Calabi-Yau pairs if their complements are minimal. 
	We then determine the Stein fillability and rational homology type of all minimal symplectic fillings for the boundary torus bundles of such $D$.
	When $ D $ is anticanonical and convex, we give explicit Betti number bounds for Stein fillings of its boundary contact torus bundle. 
\end{abstract}


{\tableofcontents}

\section{Introduction}

Let $ X $ be a smooth rational surface and let $ D\subset X $ be an effective reduced anti-canonical divisor. Such pairs $ (X,D) $ are called anti-canonical pairs and have been extensively studied since Looijenga (\cite{Lo81}). 
Gross, Hacking and Keel studied the mirror symmetry aspects of anti-canonical pairs in \cite{GrHaKe11} and \cite{GrHaKe12}. In particular, they proved Looijenga's conjecture on dual cusp singularities in \cite{GrHaKe11} and Torelli type results in \cite{GrHaKe12} conjectured by Friedman. 
The notion of symplectic log Calabi-Yau pairs, as a symplectic analogue of anti-canonical pairs, was defined by the first and the second authors in \cite{LiMa16-deformation}, where they studied different notions of deformation equivalence and enumerated minimal models.

In this paper, we present a comprehensive study of circular spherical divisors, which are generalizations of symplectic log Calabi-Yau pairs. In particular, we are interested in the following natural question on the embeddability and rigidity of circular spherical divisors.
\begin{question}\label{question}
	Given a circular spherical divisor $ D $, can we symplectically embed $ D $ into a closed symplectic 4-manifold $ (X,\omega ) $ so that $ (X,\omega, D) $ is a symplectic log Calabi-Yau pair?
\end{question}
The Looijenga conjecture can be seen as an answer to this question for negative definite divisor in terms of its dual (see Section \ref{section:neg def}). We give a complete answer (up to toric equivalence) for all other cases, i.e. when $D$ is strictly negative semi-definite or $b^+(D)\ge 1$.
Then we apply such embeddability and rigidity of circular spherical divisors to the study of symplectic fillability and the topology of minimal symplectic fillings of contact torus bundles.
In particular, we give proofs for many results previously announced in \cite{LiMa19-survey} and \cite{LiMi-ICCM}.



\subsection{Circular spherical divisors, embeddability and rigidity}
By a topological divisor, we mean a connected configuration of finitely many closed embedded oriented smooth surfaces $D=C_1 \cup \dots \cup C_r$ 
in a smooth oriented 4 dimensional manifold $ X $ (possibly with boundary or non-compact), satisfying the following conditions:  
(1) each intersection between two components is positive and transversal,
(2) no three components intersect at a common point,
and (3) $ D $ does not intersect the boundary $ \partial X $.
Since we are interested in the germ of a topological divisor, we usually omit $ X $ in the writing and just denote the divisor by $ D $. 
To a topological divisor $D$, we can associate a decorated graph called the dual graph, with vertices corresponding to the surfaces and edges corresponding to intersection points. Each vertex is decorated by the genus and self-intersection of the corresponding surface. A topological divisor is determined by its dual graph.
For each $ D $ denote by $N_D$ the neighborhood obtained by plumbing disk bundles according to its dual graph and $Y_D=\partial N_D$ the plumbed 3-manifold oriented as the boundary of $ N_D $. They are well-defined up to orientation-preserving diffeomorphisms (\cite{Ne81-calculus}).

An intersection matrix is associated to each topological divisor. For a topological divisor $ D=\cup_{i=1}^r C_i $, we denote by $[C_i]$ the homology class of $C_i$ in $ H_2(X)  $.
Note that $H_2(N_D)$ is freely generated by $[C_i]$. The intersection matrix of $D$ is the $r$ by $r$ square matrix $Q_D=(s_{ij}=[C_i]\cdot [C_{j}])$, where $\cdot$ is used for any of the pairings $H_2(X) \times H_2(X), H^2(X) \times H_2(X), H^2(X) \times H^2(X, \partial X)$. Via the Lefschetz duality for  $N_D$, the intersection matrix $Q_D$ can  be identified with the natural  homomorphism $Q_D: H_2(N_D)\to H_2(N_D, Y_D)$.  We use homology and cohomology with $\Z$ coefficient unless otherwise specified. 

In a symplectic 4-manifold $(X, \omega)$, a symplectic divisor is a topological divisor $D$ embedded in $X$, with each $C_i$  being symplectic and having the orientation positive with respect to $\omega$. 
A \textbf{symplectic log Calabi-Yau pair} $(X,D,\omega)$ is a closed symplectic 4-manifold $(X,\omega)$ together with a nonempty symplectic divisor $D=\cup C_i$ representing the  Poincare dual of $c_1(X,\omega )$.
It's an easy observation (\cite{LiMa19-survey}) that $ D $ is either a torus or a cycle of spheres. In the former case, $ (X,D,\omega ) $ is called an \textbf{elliptic log Calabi-Yau pair}. In the later case, it's called a \textbf{symplectic Looijenga pair} and it can only happen when $ (X,\omega ) $ is rational. As a consequence, we have $ b^+(Q_D)=0 $ or $ 1 $. We also remark that symplectic log Calabi-Yau pairs have vanishing relative symplectic Kodaira dimension (cf. \cite{LiZh11-relative},\cite{LiMi-logkod}).

We call a topological divisor $ D $ consisting of a cycle of spheres a \textbf{circular spherical divisor} and a \textbf{symplectic circular spherical divisor} if such $ D $ is a symplectic divisor. 
For each circular spherical divisor $ D=\cup_{i=1}^r C_i $, we associate to it a self-intersection sequence $ (s_i=[C_i]^2)_{i=1}^{r} $. 
For a circular spherical divisor, its dual graph is a cycle, which comes with cyclic and anti-cyclic symmetries. A circular spherical divisor is determined by its self-intersection sequence (up to (anti-)cyclic permutation) and given any sequence $ (s_1,\dots,s_r) $ with $ s_i\in \ZZ $, it can always be realized as the self-intersection sequence of a circular spherical divisor (via plumbing disk bundles). 
Essentially there is no difference between a sequence and a circular spherical divisor (up to (anti-)cyclic permutation of the labeling) and we denote by $D=(s_1,\dots,s_r)$ the circular spherical divisor with self-intersection sequence $(s_1,\dots,s_r)$. We also denote by $ D=(s) $ the torus with self-intersection $ s $.
This does not cause any confusion as we always require that a circular spherical divisor has length at least $ 2 $. We define the $ b^+/b^-/b^0 $ of a divisor $ D $ to be the $ b^+/b^-/b^0 $ of $ Q_D $, i.e. the number of positive/negative/zero eigenvalues. 

Circular spherical divisors generalize symplectic Looijenga pairs as they are not required to be embedded in a closed symplectic 4-manifold or to represent the Poincare dual of the first Chern class.
They are the main objects of study in this paper and we are interested in the following properties.
A circular spherical divisor $D$ is called \textbf{symplectically embeddable} if $ D $ admits a symplectic embedding into a closed symplectic 4-manifold $ (X,\omega ) $. A symplectically embeddable $D$ is called \textbf{rationally embeddable} if such $ (X,\omega ) $ can be chosen to be a symplectic rational surface, i.e. $ X\cong \CC\PP^2\# l\overline{\CC\PP}^2 $ or $S^2\times S^2$. A rationally embeddable $D$ is called \textbf{anti-canonical} if $ (X,D,\omega ) $ is a symplectic Looijenga pair for some $(X,\omega)$. An anti-canonical $D$ is called \textbf{rigid} if for any symplectic embedding of $ D$ into a symplectic 4-manifold $ (X,\omega ) $ with $ X-D $ minimal, $ (X,D,\omega ) $ is a symplectic Looijenga pair. Note that the complement of an anti-canonical $ D $ is by definition minimal.
\begin{remark}\label{rmk:not rigid}
	Observe that if $ D $ embeds into a non-rational symlectic manifold, $ D $ cannot be rigid. This is because we could assume its complement is minimal by blowing down but $ D $ cannot be anti-canonical as the ambient manifold is not rational.
\end{remark}

In terms of above terminalogy, Question \ref{question} asks for a criterion to decide whether a circular spherical divisor is anti-canonical. 
The main theorem of our paper is the following.
\begin{theorem}\label{thm:embeddable=rigid}
	For a circular spherical divisor with $ b^+(D)\ge 1 $, being symplectically embeddable, rationally embeddable, anti-canonical, and rigid are equivalent.
\end{theorem}
In particular when $b^+ (D)\ge 2$, $D$ is not symplectically embeddable (Lemma \ref{lemma:b^+ leq 1}). Combined with the classification of symplectically embeddable circular spherical divisors in Theorem \ref{thm:list}, Theorem \ref{thm:embeddable=rigid} gives a combinatorial characterization of anti-canonical circular spherical divisor and answers Question \ref{question} in the case $b^+(D)\ge 1$. In this way, Theorem \ref{thm:embeddable=rigid} can be seen as a symplectic version of Looijenga conjecture for divisors with $ b^+(D)\ge 1 $.

We show that negative semi-definite circular spherical divisors (i.e. $b^+(D)=0$) are all symplectically fillable and not rigid.
This contrasts Theorem \ref{thm:embeddable=rigid}, where divisors with $b^+\ge 1$ may not be symplectically embeddable, but once symplectically embeddable, are all rigid. 
When $D$ is strictly negative semi-definite (i.e. not negative definite), we completely determine whether $D$ is anti-canonical or not, up to toric equivalence. 
These results give us examples that are symplectically embeddable but not anti-canonical, or anti-canonical but not rigid (Example \ref{example:negative semidef} and \ref{example:negative def}). This shows that the assumption $ b^+\ge 1 $ in Theorem \ref{thm:embeddable=rigid} is sharp.

We can classify all symplectically embeddable circular spherical divisors up to toric equivalence (see Definition \ref{def:toric eq}) and list them according to the types of their boundary torus bundles. 
It is well-known that the boundary $ Y_D $ of a plumbing of a cycle of spheres $ D $ is a topological torus bundle (cf. \cite{Ne81-calculus}). Also by $-Y_D$ we mean the negative of its induced boundary orientation.
\begin{theorem}\label{thm:list}
	A circular spherical divisor with $ b^+\ge 1 $ is symplectically embeddable if and only if it is toric equivalent to one of the following:
	\begin{enumerate}[\indent $ (1) $]
		\item $ (1,p) $ or $ (-1,-p) $ with $ p=1,2,3 $, in which case $ -Y_D $ is elliptic,
		\item $ (1,1,p) $ with $ p\le 1 $, in which case $ -Y_D $ is positive parabolic,
		\item $ (0,p) $ with $ p\le 4 $, in which case $ -Y_D $ is negative parabolic,
		\item $ (1,p) $ with $ p\le -1 $ or $ (1,1-p_1,-p_2,\dots,-p_{l-1},1-p_l) $ blown-up (defined in Section \ref{subsection:rationally embeddable}) with $ p_i\ge 2 $, $ l\ge 2 $, in which case $ -Y_D $ is negative hyperbolic.
	\end{enumerate}
	In particular, all $-Y_D$ above are distinct as oriented 3-manifolds.
\end{theorem}
\begin{remark}
	The reason why we look at the negative boundary $ -Y_D $ is that induced contact structure from the divisor is a positive contact structure with respect to this negative orientation.
	For details see Section \ref{section:divisor nbhd}.
\end{remark}

\subsection{Contact structures and symplectic fillings}
A topological divisor $ D $ is called {\bf concave} (resp. {\bf convex}) if it admits a concave (resp. convex) plumbing (see Section \ref{section:divisor nbhd}).
In Proposition \ref{prop: convex-concave}, we show that a circular spherical divisor is concave (resp. convex) if and only if $ b^+(D)>0 $ (resp. $ Q_D $ negative definite).

If $ D $ has a concave (resp. convex) plumbing $ N_D $, there is a contact structure associated to $ D $ on the torus bundle obtained as the negative (resp. positive) boundary of $ N_D $, which we will denote by $ (-Y_D,\xi_D) $ (resp. $ (Y_D,\xi_D) $). This contact structure only depends on the topological divisor $ D $ and doesn't vary with the symplectic structure $ \omega $ on $ N_D $ (Proposition \ref{prop:unique-contact}). 
Furthermore, it is shown in \cite{LiMi-ICCM} (see Proposition \ref{prop:contact toric eq}) that this contact structure is invariant under toric equivalence. So toric equivalence is a natural equivalence for circular spherical divisors from the perspective of contact topology.

Symplectic fillability and Stein fillability of contact torus bundles have been extensively studied. Ding-Geiges showed in \cite{DiGe01-fillability} every torus bundle admits infinitely many weakly but not strongly symplectically fillable contact structures. Bhupal-Ozbagci constructed in \cite{BhOz14} Stein fillings for all tight contact structures on positive hyperbolic torus bundles. In \cite{DiLi18-torusbundle} Ding-Li studied symplectic fillability and Stein fillability of some tight contact structures on negative parabolic and negative hyperbolic torus bundles.
For a large family of concave circular spherical divisors (see Remark \ref{rmk:golla-lisca}), Golla and Lisca investigated the topology of Stein and minimal symplectic fillings of $ (-Y_D,\xi_D) $ in \cite{GoLi14}. Note that $ (Y_D,\xi_D) $ is always symplectic fillable if $ D $ is convex, as $ N_D $ provides a symplectic filling. In the case of elliptic log Calabi-Yau pairs, Ohta and Ono classified symplectic fillings of simple elliptic singularities up to symplectic deformation in \cite{OhOn03-simple-elliptic}. 

It turns out that, for a circular spherical divisor $ D $, $ b^+\ge 1 $ is equivalent to being concave and symplectically embeddable is equivalent to $ (-Y_D,\xi_D) $ being symplectically fillable.
Then we can use Theorem \ref{thm:embeddable=rigid} and \ref{thm:list} to determine the symplectic fillability of contact torus bundles $ (-Y_D,\xi_D) $ for all concave circular spherical divisors $ D $ and study the topology of their fillings. 
In particular, the rational homology type of minimal symplectic fillings is unique, which is completely determined by $D$. Here by \textbf{rational homology type}, we mean the Betti numbers and the intersection form on (co)homology over rational coefficients.
To describe the Betti numbers, we recall that the charge \[q(D)=12-D^2-r(D)=12-3r(D)-\sum s_i\] introduced in \cite{Fr}, which is invariant under toric equivalence. Also note that the number $b^0(D)$ of zero eigenvalues of $Q_D$ is also invariant under toric equivalence (Lemma \ref{lemma:toric eq topological}).

\begin{theorem} \label{thm:embeddable=fillable}
	Let $ D $ be a concave circular spherical divisor. Then $ (-Y_D,\xi_D) $ is symplectic fillable  if and only if $D$ is toric equivalent to one in Theorem \ref{thm:list}. Furthermore, we have that \begin{itemize}
		\item all minimal symplectic fillings have a unique rational homology type with $b_1=b_1(Y_D)-1$, $b_3=b^+=0 $, $b^0=1$ and $b^-=q(D)-2+b^0(D)$;
		\item all minimal symplectic fillings have $c_1=0$;
		\item every minimal symplectic filling is symplectic deformation equivalent to a Stein filling;
		\item there are at most finitely many (Stein) minimal symplectic fillings of $(-Y_D, \xi_D)$ up to symplectic deformation;
	\end{itemize}
	
\end{theorem}
Note that $b^\pm ,b^0$ completely determines the intersection form on (co)homology with rational coefficients. The following example of Golla-Lisca implies that integral intersection forms are not unique for fillings of every concave $D$, so our uniqueness of rational homology type is sharp.
\begin{example}[\cite{GoLi14}]
	The divisor \[D=(1,-2,-3,-3,-2,-3,-2) \] admits two different symplectic embeddings in $X=\CC\PP^2\# 9\overline{\CC\PP}^2$. Then we get two regular neighborhoods $W_1, W_2$ of the two embedded divisors and two complements $P_1,P_2$ of the neighborhoods. $P_1$ and $P_2$ are actually Stein fillings of $(-Y_D,\xi_D)$. Golla and Lisca showed the image of $H_2(P_i;\ZZ)$ in $H_2(X;\ZZ)$ are not isometric as integral lattices, which implies $P_1$ and $P_2$ have different intersection forms and thus not homotopic equivalent.
\end{example}

\begin{remark}\label{rmk:golla-lisca}
		In \cite{GoLi14}, Golla and Lisca have investigated contact torus bundles $(-Y_D,\xi_D)$ arising from $D$ in families (1), (3), (4) of Theorem \ref{thm:list} and proved that they are Stein fillable. 
		They showed that all Stein fillings of $(-Y_D, \xi_D)$ have $c_1=0$, $b_1=0$ and share the same Betti numbers.
		Moreover, up to diffeomorphism, for family (1) there is a unique Stein filling and for family (3) and (4) there are finitely many Stein fillings.
		Many of their results also hold for minimal symplectic fillings.
		Our result deals with deformation types of minimal symplectic fillings and identifies Stein fillings with minimal symplectic fillings up to deformation. 
		We improve their result by proving finiteness up to symplectic deformation, computing $ b^\pm,b^0$ and thus the rational intersection form for all minimal symplectic fillings and also generalizing the results to the new family (2) of positive parabolic contact torus bundles.
\end{remark}
\begin{remark}\label{rmk:VHM}
		In light of the above remark, Theorem \ref{thm:embeddable=fillable} is more interesting as a non-fillability criterion and in particular gives many non-fillable contact torus bundles that are not previously known. 
		In literature, contact torus bundles are mostly presented as quotients of $ T^2\times \RR $ or as Legendrian surgeries, which makes it difficult to compare these to contact torus bundles as boundaries of circular spherical divisors in our paper. We exhibit one example here. For a circular spherical divisor $D=(s_1,\dots,s_r)$ with $ s_i\ge -2 $ for all $ i $, we can construct an openbook decomposition for such contact torus bundles by \cite{LiMi-ICCM}, which corresponds to a word decomposition of the monodromy by \cite{VHM-thesis}. Then the circular spherical divisor \[
		(-1,-1,-1,\underbrace{-2,\dots,-2}_n)\sim (0,-n)
		\]
		has a word decomposition $ (aba)^{-2}a^n $. For $ n \le -5 $, this contact structure is not Stein fillable by Theorem \ref{thm:embeddable=fillable}, which is previously unknown (see \cite{VHM-thesis} Chapter 5, Example 3.(d)).
\end{remark}

When a symplectic circular spherical divisor $ D $ in $ (X,\omega ) $ is anticanonical and convex, it arises as a resolution of a cusp singularity and thus its boundary $ (Y_D,\xi_D ) $ is Stein fillable. We study the geography of its Stein fillings and give restrictions on its Betti numbers and Euler number in Proposition \ref{prop:geography}.

\textbf{Acknowledgments}: 
This paper is inspired by the results and questions in Golla-Lisca \cite{GoLi14}. 
The authors are grateful for the anonymous referees for their valuable time and constructive input.
All authors are supported by NSF grant 1611680. 


\section{Topology of circular spherical divisors}

In this section we discuss several aspects of circular spherical divisors.  We first reviews basic facts about torus bundles in Section \ref{subsection:torus bundle}.
Then in Section \ref{section:toric eq} we introduce the notion of toric equivalence for topological divisors. 
All properties we are concerned about in this paper will be invariant under toric equivalence, which makes it a useful reduction tool in the proofs.
Finally in Section \ref{subsection:homological} we give homological restrictions for a circular spherical divisor to be embedded in a closed manifold with $ b^+=1 $. In particular, Lemma \ref{lemma:topological cyclic} is essential to the proof of Theorem \ref{thm:list}.

\subsection{Boundary torus bundles}\label{subsection:torus bundle}
Denote by $ T_A $ an oriented torus bundle over $ S^1 $ with monodromy $ A\in SL(2;\ZZ) $. A torus bundle $ T_A $ is called elliptic if $ |\tr A|<2 $, parabolic if $ |\tr A|=2 $ and hyperbolic if $ |\tr A|>2 $. If $ T_A $ is parabolic or hyperbolic, we call it positive (resp. negative) if $ \tr A $ is positive (resp. negative). Also, $-T_A$ is orientation-preserving diffeomorphic to $T_{A^{-1}}$.

To describe the plumbed 3-manifold $Y_D$, we introduce the following matrix for a sequence of integers
$(t_1, \cdots, t_r)$,
\[ A(t_1,\dots,t_r)= \begin{pmatrix}
t_r & 1 \\
-1 & 0 \end{pmatrix}
\begin{pmatrix}
t_{r-1} & 1 \\
-1 & 0 \end{pmatrix} \dots 
\begin{pmatrix}
t_1 & 1 \\
-1 & 0 \end{pmatrix}  \in SL_2(\ZZ).\]

\begin{lemma}  [Theorem 6.1 in \cite{Ne81-calculus}, Theorem 2.5 in \cite{GoLi14}]   \label{lem: property of continuous fraction}
	For a circular spherical divisor $D=(s_1, ..., s_r)$, the  plumbed 3-manifold $Y_D$ is the oriented torus bundle $T_A$ 
	over $S^1$ with  monodromy  $A=A(-s_1,\dots,-s_r)$.
	Its homology is given by $H_1(Y_D;\ZZ)=\ZZ\oplus Coker(A-I)=\ZZ\oplus Coker (Q_D)$.
	The intersection matrix $Q_D$ is non-degenerate if the trace of $A(-s_1,\dots,-s_r)\ne 2$.
\end{lemma}

\subsection{Toric equivalences}\label{section:toric eq}
\begin{definition}\label{def:toric eq}
	For a topological divisor $ D=\cup C_i $, \textbf{toric blow-up} is the operation of adding a sphere component with self-intersection $ -1 $ between an adjacent pair of component $ C_i $ and $ C_{j} $ and changing the self-intersection of $ C_i $ and $ C_{j} $ by $ -1 $. \textbf{Toric blow-down} is the reverse operation.
	
	$ D^0 $ and $ D^1 $ are \textbf{toric equivalent} if they are connected by toric blow-ups and toric blow-downs. $ D $ is said to be \textbf{toric minimal} if no component is an exceptional sphere (i.e. a component of self-intersection $ -1 $).
\end{definition}
Note that toric blow-ups and blow-downs can be performed in the symplectic category by adding an extra parameter of symplectic area and are thus operations on symplectic divisors.
Also note that we could keep toric blowing down a circular spherical divisor until either it becomes toric minimal or it has length $ 2 $. 
When a circular spherical divisor has length $ 2 $, we cannot further blow it down because it would result in a non-embedded sphere and thus not a topological divisor. So we exclude this case when we talk about toric blow-down.


\begin{lemma}\label{lemma:toric eq topological}
	The  following are preserved under toric equivalence:
	\begin{enumerate}[\indent $ (1) $]
		\item $D$ being a circular spherical divisor,
		\item $ b^+(Q_D) $ and $ b^0(Q_D) $ (in particular the non-degeneracy of the intersection matrix $Q_D$),
		\item the oriented diffeomorphism type of the plumbed 3-manifold $Y_D$.  
	\end{enumerate}
	
\end{lemma} 
\begin{proof}
	(1) is obvious and (3) is part of Proposition 2.1 in \cite{Ne81-calculus}.
	
	(2) follows from a direct computation. Let $ D= (b_1,\dots,b_r) $ and $ D'=(-1,b_1-1,b_2,\dots , b_r-1) $ be its toric blow-up. Then in terms of the standard basis $ \{ E, C_1',\dots,C_r' \} $, the intersection matrix $ Q_{D'} $ is of the form \[
	\begin{pmatrix}
	-1 & 1 & 0  & \dots  & 1\\
	1 & b_1 -1  & 1 &\dots & 0\\
	0 & 1 & b_2  & \dots & \vdots \\
	\vdots & & & & 1\\
	1 & 0 & \dots & 1 & b_r-1
	\end{pmatrix}.
	\]
	Then $ Q_{D'} $ is equivalent to $ (-1)\oplus Q_D $ under the new basis $ \{ E,C_1'+E,C_2',\dots,C_{r-1}',C_r'+E \} $. So toric blow-up preserves $ b^+ $, $ b^0 $ and increases $ b^- $ by $ 1 $.
\end{proof}
Here is an example to illustrate how a self-intersection $ 0 $ component in a circular spherical divisor can be  used  to balance the self-intersection numbers of the two sides by performing a toric blow-up and a toric blow-down. 

\begin{example}\label{eg: balancing self-intersection by $0$-sphere}
	The following  three  cycles of spheres are toric equivalent: 
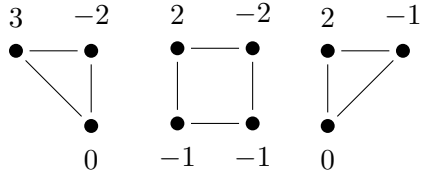
\[	\begin{tikzpicture}
	\node (x) at (0,0) [circle,fill,outer sep=5pt, scale=0.5] [label=above:$ 3 $]{};
	\node (y) at (1,0) [circle,fill,outer sep=5pt, scale=0.5] [label=above:$ -2 $]{};
	\node (z) at (1,-1) [circle,fill,outer sep=5pt, scale=0.5] [label=below:$ 0 $]{};
	\draw (x) -- (y);\draw (y) -- (z);\draw (z) -- (x);
	\end{tikzpicture}\quad
	\begin{tikzpicture}
	\node (x) at (0,0) [circle,fill,outer sep=5pt, scale=0.5] [label=above:$ 2 $]{};
	\node (y) at (1,0) [circle,fill,outer sep=5pt, scale=0.5] [label=above:$ -2 $]{};
	\node (z) at (1,-1) [circle,fill,outer sep=5pt, scale=0.5] [label=below:$ -1 $]{};
	\node (w) at (0,-1) [circle,fill,outer sep=5pt, scale=0.5] [label=below:$ -1 $]{};
	\draw (x) -- (y);\draw (y) -- (z);\draw (z) -- (w);\draw (x) to (w);
	\end{tikzpicture}\quad
	\begin{tikzpicture}
	\node (x) at (0,0) [circle,fill,outer sep=5pt, scale=0.5] [label=above:$ 2 $]{};
	\node (y) at (1,0) [circle,fill,outer sep=5pt, scale=0.5] [label=above:$ -1 $]{};
	\node (z) at (0,-1) [circle,fill,outer sep=5pt, scale=0.5] [label=below:$ 0 $]{};
	\draw (x) -- (y);\draw (y) -- (z);\draw (z) -- (x);
	\end{tikzpicture}
\]	
\end{example}
We call such a move which changes a divisor $ (\dots,k,0,p,\dots) $ to a toric equivalent divisor $ (\dots,k-n,0,p+n,\dots) $ a \textbf{balancing move} based at the $ 0 $-sphere.


\begin{lemma}\label{lem: not negative semi-definite => at least one non-negative}
	Any circular spherical divisor is toric equivalent to a toric minimal one or $(-1, p)$. If $D=(-1,p)$, then $Q_D$ is degenerate only if $p=-4$.

	Suppose $D=(s_1,\dots,s_r)$ is  a toric minimal cycle of spheres. Then
	
	\begin{enumerate}[\indent$  (1) $]
		\item $b^+(Q_D)\geq 1$    if and only if  $s_i\geq 0$ for some $i$.
		
		\item $Q_D$ is  negative definite if 
		$s_i\leq -2$   for all $i$ and less than $-2$ for some $i$.
		$Q_D$ is  negative semi-definite but not negative definite if 
		$s_i= -2$ for each $i$.
		\item $Q_D$ is non-degenerate 
		if either  $s_1\geq 0$ and $s_i\leq -2$ for $i\geq 2$, or    $s_1=s_2=0$ and $s_i\leq -2$ for $i\geq 3$.
	\end{enumerate} 
\end{lemma}

\begin{proof}
	By toric blow-down, any circular spherical divisor is toric equivalent to a toric minimal one or one of length $ 2 $. If $ D $ has length $ 2 $ and not toric minimal, then it is of the form $ (-1,p) $. Then $ \det(Q_D)=-p-4=0 $ only if $ p=-4 $.
	
	(1) and (2) are well-known (cf. Lemma 8.1 in \cite{Ne81-calculus}, Lemma 2.5 in \cite{GaMa13-LF}).  
	
	To  prove (3),  first notice that for $ r=2 $, $ Q_D=\begin{pmatrix}
	s_1 & 1\\ 1 & s_2
	\end{pmatrix} $, so $ Q_D $ is non-degenerate if $ s_1\ge 0 $ and $ s_2\le -2 $. 
	For $ r\ge 3 $, by Lemma \ref{lem: property of continuous fraction}, 
	we just need to show that the trace of the monodromy matrix  is not equal to $2$.
	
	For this purpose, we recall the following observation in Lemma 5.2 in \cite{Ne81-calculus}:
	Suppose $t_i \le -2$ for all $i=1,\dots,r$ and define $ p_j,q_j\in \ZZ $ recursively by \begin{align*} 
		p_{-1}=0,p_0=1,p_{i+1}=-t_{i+1}p_i-p_{i-1};\\
		q_{-1}=-1,q_0=0,q_{i+1}=-t_{i+1}q_i-q_{i-1}. \end{align*}
	Then $$ A(-t_1,\dots,-t_j)=  \begin{pmatrix}
	p_j & q_j \\
	-p_{j-1} & -q_{j-1}\end{pmatrix} ,$$
	with  $p_j \ge p_{j-1}+1 \ge 0$ and $q_j \ge q_{j-1}+1 \ge 0$ for all $ j $.
	
	Also notice that $p_j \ge q_j+1$ for all $j$, which can be verified by induction:
	Let $ \Delta_j=p_j-q_j $. Since $ \Delta_{-1}=\Delta_0=1 $, it suffices to show $ \Delta_{j+1}\ge \Delta_j $. It follows from the fact that $ \Delta_{j+1}-\Delta_j=(-t_{j+1}-1)\Delta_j-\Delta_{j-1}\ge \Delta_j-\Delta_{j-1} $ as $ t_{j+1}\le -2 $.

	We  apply this observation   to  the chain of spheres with negative self-intersection.
	In the first case where $ (s_1,\dots,s_r)=(s_1,t_1,\dots,t_{r-1}) $ with $ s_1\ge 0 $ and $ t_i\le -2 $,  the monodromy matrix is 
	\[A(-s_1,-t_1\dots,-t_{r-1})= A(-t_1, \cdots, -t_{r-1}) \begin{pmatrix} -s_1 &1\\-1 &0\end{pmatrix}   =\begin{pmatrix}
	-s_1p_{r-1}-q_{r-1} & p_{r-1} \\
	s_1p_{r-2}+q_{r-2} & -p_{r-2}\end{pmatrix},
	\]
	with $p_r, q_r$ defined as in the observation above.  The trace is $ -s_1p_{r-1}-q_{r-1}-p_{r-2}\le -q_{r-2}-p_{r-2}-1\le -2(q_{r-2}+1)\le -2 $, so
	not equal to $2$. 
	In the second case where $ (s_1,\dots,s_r)=(0,0,t_1,\dots,t_{r-2}) $ , the monodromy matrix is  
	\[A(0,0,  -t_1, \dots,-t_{r-2})=\begin{pmatrix}
		p_{r-2} & q_{r-2} \\
	-p_{r-3} & -q_{r-3}
	\end{pmatrix} \begin{pmatrix}
		0 & 1\\-1 & 0
	\end{pmatrix}\begin{pmatrix}
		0 & 1\\-1 & 0
	\end{pmatrix}= \begin{pmatrix}
	-p_{r-2}
	& -q_{r-2}\\
	p_{r-3}
	& q_{r-3} \end{pmatrix}.\] 
	The trace is $ -p_{r-2}+q_{r-3}\le -p_{r-2}+q_{r-2}-1\le -2 $, so again cannot be $ 2 $.
\end{proof}




\subsection{Homological restrictions when $ b^+(X)=1 $}\label{subsection:homological}
In this subsection, we assume $ D=\cup C_i $ is a circular spherical divisor. Let $ r(D) $ denote the number of components of $ D $ and $r^{\geq 0}(D)$ the number of components with non-negative self-intersection.
Here are some restrictions on homologous components of $D$. 

\begin{lemma}  \label{lemma:homologous components}  For any $D$ embedded in a smooth 4-manifold $ X $, we have the following \begin{enumerate}[\indent $ (1) $]
		\item At most three components   are homologous in $X$. There are three homologous components only if $r(D)=3$. 
		\item There are a pair of homologous components only if  $r(D)\leq 4$.
		\item If $[C_i]=[C_{i+1}]$ for some $i$  then $r(D)=3, 
		s_i=s_{i+1}=1$, or $r(D)=2, s_i=s_{i+1}=2$. 
	\end{enumerate}
\end{lemma}
\begin{proof}
	Suppose there are $ m $ components homologous to $ a\in H_2(X) $ in $ D $. Note that $ a^2\in \{ 0, 1,2\} $ because the divisor has only one cycle and the components are required to intersect positively and transversally. If $ a^2=1 $, then these components are all adjacent. In order to form exactly one cycle, $ m $ is at most $ 3 $. In particular, when $ m=3 $, there cannot be other components, i.e. $ r(D)=3 $. If $ a^2=0 $, there is another component $ C $ intersecting all these components. Again in order to form exactly one cycle, $ m $ is at most $ 2 $. Similarly if $ a^2=2 $, they are adjacent and we must have $ m=r(D)=2 $. This proves (1) and (3).
	
	Suppose $ C_i,C_j $ are a pair of homologous components in $ D $. If they are adjacent, then $ [C_i]\cdot [C_j]=1 $ or $ 2 $ and $ r(D)\le 3 $ by the above discussion. If they are not adjacent, then any other component intersecting $ C_i $ must also intersect $ C_j $. There must be exactly two such components to form a cycle. So $ r(D)=4 $ and this proves (2).
\end{proof}

Note that the above restrictions hold locally. When $ X $ is closed with $b^+(X)=1$, there are various restrictions on components with non-negative self-intersection.

\begin{lemma} \label{lem: non-negative components}
	Suppose $D$ is embedded in a closed manifold $X$ with  $b^+(X)=1$. 
	\begin{enumerate}[\indent $ (1) $]
		\item If $C_i$ and $C_j$ are not adjacent and $s_i\geq 0, s_j\geq 0$, then $[C_i]=\pm [C_j]$ and $s_i=s_j=0$. 
		\item $r^{\geq 0}(D)\leq 4$. 
		\item $r^{\geq 0}(D)=4$ only if $r(D)=4, s_i=0$ for each $i$ and $[C_1]=[C_3], [C_2]=[C_4]$. 
		\item Suppose $r(D)\geq 3$. If  $s_i\geq 1, s_{i+1}\geq 1$ for some $i$, then $[C_i]=[C_{i+1}]$ and $s_i=s_{i+1}=1$. This is only possible when $r(D)=3$. 	
	\end{enumerate}
\end{lemma}

\begin{proof}
	Since $b^+(X)=1$, by the light cone lemma (cf. \cite{McDuffSalamon1996}), any two disjoint components with non-negative self-intersection must be homologous up to sign and have self-intersection $0$. 
	
	(2) and (3) follow from the (1). 
	
	
	For (4), we can assume the two spheres are $C_1$ and $C_2$. 
	Since $r(D)\geq 3$, we have $[C_1]\cdot[C_2]=1$. By toric blowing up the intersection point between $C_1$ and $C_2$, we get two disjoint spheres with classes $ [C_1']=[C_1]-E $ and $ [C_2']=[C_2]-E $, where $ E $ is the exceptional class and $ [C_1']^2=[C_1]^2-1\ge 0 $, $ [C_2']^2=[C_2]^2-1\ge 0 $. Then by (1), we have $ [C_1']=[C_2'] $ with $ [C_1']^2=0 $ and thus $ [C_1]=[C_2] $ with $ [C_1]^2=1 $. The fact that $ r(D)=3 $ follows from (3) in Lemma \ref{lemma:homologous components}.
	
\end{proof}
\begin{remark}
	Note in (1) of Lemma \ref{lem: non-negative components}, if $ C_i $ and $ C_j $ are symplectic spheres, we would have $ [C_i]=[C_j] $.
\end{remark}
\begin{lemma}\label{lemma:topological cyclic}
	Suppose $D$ has $ b^+(Q_D)= 1 $ and is embedded in a closed manifold $X$ with  $b^+(X)=1$. Let $ k,p,p_1,p_2 $ be integers such that $ k\ge 0 $ and $ p,p_1,p_2<0 $. Up to cyclic and anti-cyclic permutations of $D$, we have the following.
	
	\begin{enumerate}[\indent$  (1) $]
		\item If $r(D)\geq 5$, then $r^{\geq 0}(D)\leq 2$. When $r^{\geq 0}(D)=2$,    $s_1\geq 0, s_2=0$. 
		
		\item If $r(D)=4$ and  $r^{\geq 0}(D)\geq 3$, then  $D=(k, 0, p, 0), k+p\leq 0$ and $  [C_2]=[C_4] $.
		
		\item If  $r(D)=4$ and $r^{\geq 0}(D)=2$, then the only possibilities of $D$ are\\
		(i) $(0, p_1, 0, p_2), [C_1]=[C_3]$, \\(ii) $(k, 0, p_1, p_2), p_1+p_2+k\leq 0$. 
		
		\item If  $r(D)=3$ and $r^{\geq 0}(D)=3$, then the only possibilities of $D$ are \\(i) $(1, 1, 1), [C_1]=[C_2]=[C_3]$,  \\(ii) $(1, 1, 0), [C_1]=[C_2]$, \\(iii)  $(k, 0, 0)$, $ k\le 2 $.
		
		\item If $r(D)=3$ and $r^{\geq 0}(D)=2$, then the only possibilities of  $D$ are \\(i)  $(1, 1, p), [C_1]=[C_2]$,  \\(ii)  $(k , 0, p), p+k\leq 2$.
		
		\item If  $r(D)=2$ and  $r^{\geq 0}(D)=2$, then $D$ is one in family $ \mc{F}(2,2)= $\\
		$\{(4, 1), (4, 0), (3, 1), (3,0),  (2, 2), (2, 1), (2, 0), 
		(1, 1), (1, 0), (0, 0)\}$.
		
		\item If  $r(D)=2$ and  $r^{\geq 0}(D)=1$, then  $D=(k, p)$.
		
		\item If  $r(D)=2$ and  $r^{\geq 0}(D)=0$, then   $D$ is one in family $ \mc{F}(2,0)= $ \\
		$\{(-1, -1),  (-1, -2), (-1, -3) \}$.
	\end{enumerate}
	
\end{lemma}
\begin{proof}
	\textbf{Case (1):} Suppose $r(D)\geq 5$. If   $r^{\geq 0}(D)\geq 3$ then two such components are not adjacent. But this is impossible due to the (1) of  Lemma \ref  {lem: non-negative components} and the (3) of Lemma \ref  {lemma:homologous components}. 
	Hence $r^{\geq 0}(D)\leq 2$ in this case. 
	When $r^{\geq 0}(D)=2$, the two components must be adjacent by the same reasoning. The claim that one of them has self-intersection $0$ follows from (4) of Lemma \ref  {lem: non-negative components} and the (3) of  Lemma \ref{lemma:homologous components}.
	
	\textbf{Case (2):} The proof is similar when $r(D)=4$ and $r^{\geq 0}(D)\geq 3$.   In this case  two such components are not adjacent, say $C_2, C_4$. By the (1) of  Lemma \ref  {lem: non-negative components}, $[C_2]=[C_4]$, $s_2=0=s_4$.  
	
	\textbf{Case (3):} Suppose  $r(D)=4$ and $r^{\geq 0}(D)=2$. 
	If two such components are not adjacent,  we can assume them to be $C_1, C_3$, 
	which satisfy $[C_1]=[C_3]$ and $s_1=s_3=0$ by the (1) of  Lemma \ref  {lem: non-negative components}.  If the two components are adjacent, 
	we can assume them to be
	$C_1, C_2$. Notice that $[C_1]\ne [C_2]$ due to the (3) of Lemma \ref  {lemma:homologous components}.  Now it follows from the 4th bullet of Lemma \ref  {lem: non-negative components} that 
	either $s_1=0$ or $s_2=0$. 	
	
	\textbf{Case (4):} Suppose $r(D)=3=r^{\geq 0}(D)$. Since $s_i\geq 0$ for any $i$,  It is easy to see (i), (ii), (iii) give all the possibilities by (4) of Lemma \ref  {lem: non-negative components}. It's easily checked by hand that $ (k,0,0) $ has $ b^+\ge 2 $ when $ k\ge 3 $.

	\textbf{Case (5):} If $r^{\geq 0}(D)=2$, apply (4) of Lemma \ref  {lem: non-negative components} to the pair of components $C_i, C_j$ with $s_i\geq 0, s_j\geq 0$. 
	
	\textbf{Case (6)(7)(8):} Suppose $r(D)=2$.  Then we just check that the determinant of $Q_D=s_1 s_2-4 \leq 0$.

	
	
\end{proof}
%

\section{Embeddability and rigidity when $b^+(D)\ge 1$}\label{section:rigid}

This section is devoted to the proof of Theorem \ref{thm:embeddable=rigid} and \ref{thm:list}. 
We start with the following observation on the embeddability of circular spherical divisors with $b^+\ge 2$.
\begin{lemma}\label{lemma:b^+ leq 1}
	A topological circular spherical divisor $ D $ cannot be symplectically embedded in a closed symplectic 4-manifold if $ b^+(D)\ge 2 $.
\end{lemma}
\begin{proof}
	Toric blow down $ D $ to $ \overline{D}\subset (\overline{X},\overline{\omega}) $ so that $ \overline{D} $ is either $ (-1,p) $ or toric minimal and $ b^+(\overline{D})=b^+(D) $ by Lemma \ref{lemma:toric eq topological}.
	Note that a divisor of form $ (-1,p) $ always has $ b^+\le 1 $. Now suppose $ \overline{D} $ is toric minimal. 
	Then there is at least one component $ C_i $ in $ D $ with $ C_i^2\ge 0 $ by Lemma \ref{lem: not negative semi-definite => at least one non-negative}, which implies $ (X,\omega ) $ is rational or ruled (\cite{Mc90-structure}).
	This implies $ b^+(D)\le 1 $.
\end{proof}
As a result, there is no symplectically embeddable circular spherical divisor with $ b^+\ge 2 $ and we only need to consider the case with $b^+=1$. Next we show that the several properties of circular spherical divisors are preserved under toric equivalence, so it suffices to consider the toric minimal divisors.

\begin{lemma}\label{lemma:toric eq preserves}
	A circular spherical divisor $ D $ being symplectically embeddable, rationally embeddable, anti-canonical or rigid is preserved under toric equivalence.
\end{lemma}
\begin{proof}
	Since toric blow-ups and blow-downs can be realized by symplectic blow-ups and blow-downs when the divisor is symplectic and symplectic blow-up or blow-down of a symplectic rational surface is still rational, it's clear that being symplectically embeddable and rationally embeddable is preserved.
	
	Let $ (X,D,\omega ) $ be a symplectic Looijenga pair. Blow up at a transverse intersection point of $ D $ to get $ (X',\omega') $ with the natural inclusion $ \iota_*:H_2(X;\ZZ )\to H_2(X';\ZZ ) $. Denote by $ D' $ the union of the proper transform of $ D $ and the exceptional curve $ E $. Then $ D' $ is a toric blow-up of $ D $ with $ [D']=[D]-[E] $. So $ [D']=\iota_*[D]-[E]=\iota_*(-K_X)-[E]=-K_{X'} $ and $ (X',D',\omega') $ is also a symplectic Looijenga pair. The proof for toric blow-down is the same but goes backwards.
	
	

	Let $ D $ be a circular spherical divisor and $ D' $ a toric blow-up of $ D $ with exceptional component $ E $. Suppose $ D $ is rigid. For any symplectic embedding of $ D' $ into $ (X',\omega') $, we can blow-down $ E $, which is a symplectic exceptional sphere in $ (X',\omega') $, to get a symplectic embedding of $ D $ into $ (X,\omega) $. If $ X'-D' $ is minimal, then $ X-D $ is also minimal, since any exceptional curve away from $ D $ would lift to an exceptional curve in $ X'-D' $. Now that $ D $ is anti-canonical in $ (X,\omega) $, we have $ D' $ is anti-canonical in $ (X',\omega') $ by the previous paragraph. So $ D' $ is rigid. The same argument goes backwards and proves that $ D $ is rigid if $ D' $ is rigid.
\end{proof}
Notice that we have the following sequence of implications simply by their definitions.\[
\text{symplectically embeddable} \Leftarrow \text{rationally embeddable} \Leftarrow \text{anti-canonical} \Leftarrow \text{rigid}
\]
To prove Theorem \ref{thm:embeddable=rigid}, it suffices to prove the converse for every arrow above. 
We recall facts about deformation classes of symplectic log Calabi-Yau pairs in Section \ref{section:log CY} and recollect some tools from pseudoholomorphic curves in Section \ref{section:maximal surface}. These will be useful also in later sections.
The proofs are distributed in Section \ref{subsection:symp embed = rational embed} through \ref{subsection:anti-canonical = rigid}, as indicated by the corresponding titles. The proof of Theorem \ref{thm:list} is contained in Section \ref{subsection:rationally embeddable}.

\subsection{Symplectic log Calabi-Yau pairs}\label{section:log CY}
Recall that a symplectic log Calabi-Yau pair $(X,D,\omega)$ is a closed symplectic 4-manifold $(X,\omega)$ together with a nonempty symplectic divisor $D=\cup C_i$ representing the  Poincare dual of $c_1(X,\omega )$.
It is called a symplectic Looijenga pair if $D$ is a cycle of spheres.
In this section, we review some facts about deformation classes of symplectic log Calabi-Yau pairs studied in \cite{LiMa16-deformation}. In particular, Theorem \ref{thm:minimal model} gives a list of anti-canonical circular spherical divisors in $ \CC\PP^2,S^2\times S^2 $ and $ \CC\PP^2\#\overline{\CC\PP^2} $, and Corollary \ref{cor: finite deformation} serves as the source of finiteness for symplectic fillings in Theorem \ref{thm:embeddable=fillable}.

We have introduced toric blow-ups and blow-downs in Definition \ref{def:toric eq}. Here we introduce another pair of operations on symplectic divisors.
\begin{definition}
	A {\bf non-toric blow-up} of $D$ is the  proper transform of a symplectic blow-up centered  at a smooth point of $D$.
	A  {\bf non-toric blow-down} is the reverse operation which symplectically blows down an exceptional sphere not contained in $ D $.
\end{definition}
Both toric and non-toric blow-ups/downs preserve the log Calabi-Yau condition and have analogues in the holomorphic category.

Through a maximal sequence of non-toric blow-downs and then a maximal sequence of toric blow-downs, 
we can get a symplectic Looijenga pair $ (\overline{X},\overline{D},\overline{\omega}) $ from any symplectic Looijenga pair $ (X,D,\omega ) $ such that $ \overline{X}=\CC\PP^2,S^2\times S^2 $ or $ \CC\PP^2\#\overline{\CC\PP}^2 $. 

Various notions of equivalences have been introduced in the study of symplectic deformation classes of symplectic log Calabi-Yau pairs in \cite{LiMa16-deformation}, the following one is related to symplectic deformation of fillings.
\begin{definition}
	
	Let $(X^0,D^0,\omega^0)$ and $(X^1, D^1,\omega^1)$ be pairs of 4-manifolds and symplectic divisors in them. When $X^0=X^1$,  they are said to be 
	\textbf{symplectic homotopic}  if $(D^0, \omega^0)$ and $(D^1, \omega^1)$  are connected by a family of symplectic divisors $(D^t, \omega^t)$.
	$(X^0, D^0, \omega^0)$ and $(X^1, D^1, \omega^1)$  are said to be 
	\textbf{symplectic deformation equivalent} if they are symplectic homotopic,  up to an orientation preserving  diffeomorphism. 
\end{definition}

We recall here the deformation classes of symplectic  Looijenga pairs with the ambient manifold being minimal or $ \CC\PP^2\#\overline{\CC\PP}^2 $, all of them having length less than $5$. 
\begin{theorem}[\cite{LiMa16-deformation}, Theorem 2.4 in \cite{Fr}]\label{thm:minimal model}
	Any symplectic log Calabi-Yau pair $ (X,D,\omega) $ with $ X=\CC\PP^2,S^2\times S^2 $ or $ \CC\PP^2\#\overline{\CC\PP}^2 $ is symplectic deformation equivalent to one of the following. Each of them is realized by a K\"ahler pair. In particular, all circular spherical divisors listed below are anti-canonical.
	
	
	$\bullet$  Case $(B)$: $X=\mathbb{CP}^2$,
	$c_1=3h$.

	$(B1)$ $D$ is a torus, 
	
	$(B2)$ $D$ consists of a $h-$sphere and a $2h-$sphere, or
	
	$(B3)$ $D$ consists of three $h-$spheres.
	The graphs in (B1), (B2), and (B3) are given respectively  by
	
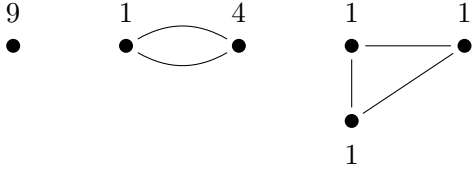
\[
	\begin{tikzpicture}
	\node (x) at (-0.5,0) [circle,fill,outer sep=5pt, scale=0.5] [label=above:$ 9 $]{};
	
	\node (x1) at (1,0) [circle,fill,outer sep=5pt, scale=0.5] [label=above:$ 1 $]{};
	\node (y1) at (2.5,0) [circle,fill,outer sep=5pt, scale=0.5] [label=above:$ 4 $]{};
	\draw (x1) to[bend right] (y1);\draw (y1) to[bend right] (x1);
	
	\node (x2) at (4,0) [circle,fill,outer sep=5pt, scale=0.5] [label=above:$ 1 $]{};
	\node (y2) at (5.5,0) [circle,fill,outer sep=5pt, scale=0.5] [label=above:$ 1 $]{};
	\node (z2) at (4,-1) [circle,fill,outer sep=5pt, scale=0.5] [label=below:$ 1 $]{};
	\draw (x2) -- (y2);\draw (y2) -- (z2);\draw (z2) -- (x2);
	\end{tikzpicture}
	\]	

	$\bullet$ Case $(C)$: $X=S^2 \times S^2$, $c_1=2f_1+2f_2$, where $f_1$ and $f_2$ are the homology classes
	of the two factors. 
	
	$(C1)$ $D$ is a torus. 
	
	$(C2)$ $r(D)=2$ and  $[C_1]=bf_1+f_2, [C_2]=(2-b)f_1+f_2$.
	
	$(C3)$ $r(D)=3$ and  $[C_1]=bf_1+f_2, [C_2]=f_1, [C_3]=(1-b)f_1+f_2$.
	
	$(C4)$ $r(D)=4$ and $[C_1]=bf_1+f_2,  [C_2]=f_1, [C_3]=-bf_1+f_2, [C_4]=f_1$.
	
	The graphs in (C1), (C2), (C3) and (C4) are given respectively  by
	
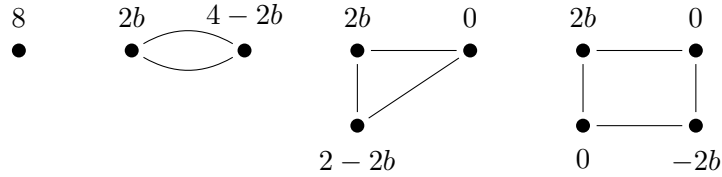
\[
	\begin{tikzpicture}
	\node (x) at (-0.5,0) [circle,fill,outer sep=5pt, scale=0.5] [label=above:$ 8 $]{};
	
	\node (x1) at (1,0) [circle,fill,outer sep=5pt, scale=0.5] [label=above:$ 2b $]{};
	\node (y1) at (2.5,0) [circle,fill,outer sep=5pt, scale=0.5] [label=above:$ 4-2b $]{};
	\draw (x1) to[bend right] (y1);\draw (y1) to[bend right] (x1);
	
	\node (x2) at (4,0) [circle,fill,outer sep=5pt, scale=0.5] [label=above:$ 2b $]{};
	\node (y2) at (5.5,0) [circle,fill,outer sep=5pt, scale=0.5] [label=above:$ 0 $]{};
	\node (z2) at (4,-1) [circle,fill,outer sep=5pt, scale=0.5] [label=below:$ 2-2b $]{};
	\draw (x2) -- (y2);\draw (y2) -- (z2);\draw (z2) -- (x2);
	
	\node (x3) at (7,0) [circle,fill,outer sep=5pt, scale=0.5] [label=above:$ 2b $]{};
	\node (y3) at (8.5,0) [circle,fill,outer sep=5pt, scale=0.5] [label=above:$ 0 $]{};
	\node (z3) at (8.5,-1) [circle,fill,outer sep=5pt, scale=0.5] [label=below:$ -2b $]{};
	\node (w3) at (7,-1) [circle,fill,outer sep=5pt, scale=0.5] [label=below:$ 0 $]{};
	\draw (x3) -- (y3);\draw (y3) -- (z3);\draw (z3) -- (w3);\draw (w3) to (x3);
	\end{tikzpicture}
	\]	

	$\bullet$ Case $(D)$: $X=\mathbb{C}P^2 \# \overline{\mathbb{C}P^2}$, $c_1=f+2s$, where $f$ and $s$ are the fiber class and section class 
	with $f\cdot f=0$, $f\cdot s=1$ and $s\cdot s=1$.

	$(D1)$ $D$ is a torus.
	
	$(D2)$ $r(D)=2$,  and either 
	$([C_1],[C_2])=(af+s,(1-a)f+s)$ or $([C_1],[C_2])=(2s, f)$.
	
	$(D3)$  $r(D)=3$ and  $[C_1]=af+s, [C_2]=f, [C_3]=-af+s$.
	
	$(D4)$ $r(D)=4$ and  $[C_1]=af+s, [C_2]=f,  [C_3]=-(a+1)f+s, [C_4]=f$.
	
	The graphs in (D2), (D3) and (D4) are given respectively by
	
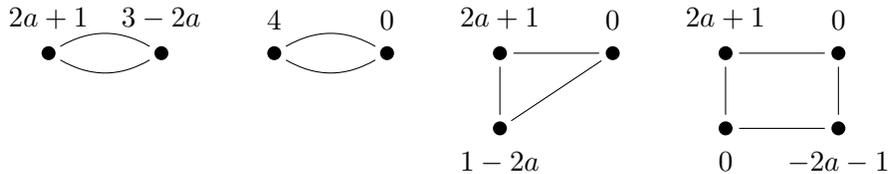
\[
	\begin{tikzpicture}
	\node (x) at (-2,0) [circle,fill,outer sep=5pt, scale=0.5] [label=above:$ 2a+1 $]{};
	\node (y) at (-0.5,0) [circle,fill,outer sep=5pt, scale=0.5] [label=above:$ 3-2a $]{};
	\draw (x) to [bend right] (y);\draw (y) to [bend right] (x);
	
	\node (x1) at (1,0) [circle,fill,outer sep=5pt, scale=0.5] [label=above:$ 4 $]{};
	\node (y1) at (2.5,0) [circle,fill,outer sep=5pt, scale=0.5] [label=above:$ 0 $]{};
	\draw (x1) to[bend right] (y1);\draw (y1) to[bend right] (x1);
	
	\node (x2) at (4,0) [circle,fill,outer sep=5pt, scale=0.5] [label=above:$ 2a+1 $]{};
	\node (y2) at (5.5,0) [circle,fill,outer sep=5pt, scale=0.5] [label=above:$ 0 $]{};
	\node (z2) at (4,-1) [circle,fill,outer sep=5pt, scale=0.5] [label=below:$ 1-2a $]{};
	\draw (x2) -- (y2);\draw (y2) -- (z2);\draw (z2) -- (x2);
	
	\node (x3) at (7,0) [circle,fill,outer sep=5pt, scale=0.5] [label=above:$ 2a+1 $]{};
	\node (y3) at (8.5,0) [circle,fill,outer sep=5pt, scale=0.5] [label=above:$ 0 $]{};
	\node (z3) at (8.5,-1) [circle,fill,outer sep=5pt, scale=0.5] [label=below:$ -2a-1 $]{};
	\node (w3) at (7,-1) [circle,fill,outer sep=5pt, scale=0.5] [label=below:$ 0 $]{};
	\draw (x3) -- (y3);\draw (y3) -- (z3);\draw (z3) -- (w3);\draw (w3) to (x3);
	\end{tikzpicture}
	\]	
	%
	
\end{theorem}

Since each pair in Theorem \ref{thm:minimal model} deforms to  a K\"ahler pair 
and blow-up/down can be performed in the K\"ahler category, these K\"ahler pairs blow-up to K\"ahler representatives in each symplectic deformation class. So we have the following lemma.
\begin{lemma}[\cite{LiMa19-survey}] \label{thm: symplectic deformation class=homology classes}
	Each symplectic deformation class  contains a K\"ahler pair. 
	

	
\end{lemma}

In the holomorphic category, We have the following finiteness of deformation classes in \cite{Fr}.
\begin{theorem}[Theorem 3.1 in \cite{Fr}] \label{thm: finite kahler deformation}
	There are only finitely many deformation types of anti-canonical pairs with the same self-intersection sequence. 
\end{theorem}

Combining Lemma \ref{thm: symplectic deformation class=homology classes} and Theorem \ref{thm: finite kahler deformation}, we obtain the following finiteness of symplectic deformation classes.
\begin{cor}[\cite{LiMa19-survey}] \label{cor: finite deformation} 
	There are only finitely many symplectic deformation types of symplectic log Calabi-Yau  pairs with the same self-intersection sequence.
\end{cor}
%

\subsection{Pseudo-holomorphic curves in dimension 4}\label{section:maximal surface}
In this subsection, we collect some useful notions in the theory pseudo-holomorphic curves in dimension 4 (\cite{McOp13-nongeneric}). These will be used frequently in the rest of the paper to study the minimality of divisor complements.


\begin{lemma}\label{lemma:maximal=minimal complement}
	Suppose $F\subset (X,\omega)$ is an embedded symplectic surface without sphere components. Then $X-F$ is minimal if and only if $ [F]\cdot e\neq 0 $ for any exceptional class $ e $.
\end{lemma}
\begin{proof}
	By Proposition 4.1 of \cite{OhtOn05-cusp}, for any exceptional class $e$, there exists an almost complex structure $J$ such that both $F$ and an embedded representative $S$ of $e$ are $J$-holomorphic. By positivity of intersection, we have $ [F]\cdot e\ge 0 $. In particular, $ [F]\cdot e=0 $ if and only if $ F $ and $ S $ are disjoint.
\end{proof}

Similar to Proposition 4.1 of \cite{OhtOn05-cusp}, McDuff and Opshtein gave a criterion on the existence of embedded pseudo-holomorphic curves relative to a pseudo-holomorphic normal crossing divisor. 
\begin{definition}
	Let $ D=\cup C_i $ be an $ \omega- $orthogonal symplectic divisor in $ (X,\omega ) $. An exceptional class $ e\in H_2(X;\ZZ ) $ is called \textbf{$ D- $good} if $ e\cdot [C_i]\ge 0 $ for all $ i $.
\end{definition}
\begin{lemma}[Theorem 1.2.7 of \cite{McOp13-nongeneric}]\label{lem: existence of nice J-sphere}
	Let $D$ be an $\omega$-orthogonal symplectic divisor.
	There is a non-empty space $\mathcal{J}(D)$ of $\omega$-tamed almost complex structures making $D$ pseudo-holomorphic such that
	for any $ D- $good exceptional class $e$, there is a residual subset $\mathcal{J}(D,e) \subset \mathcal{J}(D)$
	so that $e$ has an embedded $J$-holomorphic representative for all $J\in \mathcal{J}(D,e)$. 
\end{lemma}
So for a $ D $-good exceptional class $ e $, we have $ e\cdot [D]\ge 0 $. In particular, $ e $ is $ D-$good if $ e\cdot [C_i]=0 $ for every component $ C_i $ of $ D $. So an exceptional curve $ E $ is disjoint from $ D $ if and only if $ [E]\cdot [C_i]=0 $ for all $ i $.

We also recall that a homology class $ b\in H_2(X;\ZZ ) $ is said to be stable if $ b $ for any $\omega$-tame almost complex structure $ J$, it can be represented by a $ J $-holomorphic curve. In particular, by \cite{Mc90-structure} we see that any exceptional class $ e $ is stable. Then any symplectic surface class of non-negative self-intersection pairs non-negatively with a stable class.

\subsection{Symplectically embeddable $D$ is rationally embeddable}\label{subsection:symp embed = rational embed}
Actually we only need to consider circular spherical divisors with $ b^+= 1 $ in a symplectic rational surface by the following lemma.
\begin{lemma}\label{lemma:rational-embed}
	Let $ D $ be a symplectic circular spherical divisor in $ (X,\omega ) $ with $ b^+(Q_D)= 1 $, then $ (X,\omega ) $ is rational.
\end{lemma}
\begin{proof}
	Since being a symplectic rational surface is preserved under blow-up and blow-down, we could assume that $ D $ is toric minimal or of the form $ (-1,p) ,p>-4$. If $ s_i\ge 1 $ for some $ i $, then $ (X,\omega ) $ is rational as it contains a positive symplectic sphere (\cite{Mc90-structure}).
	
	Now we assume $ D $ is toric minimal and $ s_i\le 0 $ for all $ i $. By Lemma \ref{lem: not negative semi-definite => at least one non-negative}, we must have $ s_i=0 $ for some $ i $ in order for $ b^+(Q_D)\ge 1 $ and thus $ (X,\omega ) $ must rational or ruled. Without loss of generality, we assume $ s_1=0 $. If $ X $ is irrational ruled with $ \pi:X\to B $, then $ [C_1] $ must be the fiber class. So $ [C_2] $ must contain a positive multiple of the section class as $ [C_1]\cdot [C_2]\ge 1 $. Then $ \pi|_{C_2}:C_2\to B $ has positive degree so that $ g(C_2)\ge g(B)\ge 1 $, which is contradiction.
	
	Suppose $ D $ is of the form $ (-1,p) ,p>-4$. When $ p\ge 0 $, it follows from the same argument as above. The case $ (-1,\epsilon-2) $, $ \epsilon=-1,0,1 $, needs a different argument. Let $ D=C_1\cup C_2 $ with $ [C_1]^2=-1 $, $ [C_2]^2=\epsilon-2 $ and $ [C_1]\cdot [C_2]=2 $. Blow down $ C_1 $ to get $ X' $ such that $ X=X'\# \overline{\CC\PP}^2 $ and $ C_2 $ becomes an immersed nodal symplectic sphere $ C_2' $ with self-intersection $ 2+\epsilon $. Smoothing the singularity of $ C_2' $ we obtain a smoothly embedded symplectic torus $ T $ with self-intersection $ 2+\epsilon\ge 1 $. By Proposition 4.3 of \cite{LiMaYa14-CYcap}, $ (X,\omega ) $ is rational or ruled. Suppose $ X $ is irrational ruled, then $ H_2(X;\ZZ) $ is generated by $ \{ f,s,e_1,\dots,e_k  \} $ where $ f $ is the class of a fiber, $ s $ is the class of a section and $ e_i $'s are exceptional classes. 
	As an exceptional class, $ [C_1] $ must be of the form $ e_i $ or $ f-e_i $ by Corollary 5.C in \cite{Biran-packing}.
	Since $ C_2 $ is an embedded symplectic sphere, by Lemma 6.1 of \cite{seppi-li-wu-stability}, we have $ [C_2]=bf+\sum \pm e_i $ and thus $ [C_1]\cdot [C_2]=\pm 1 $, which is a contradiction.
	
\end{proof}

\subsection{Classify rationally embeddable $D$}\label{subsection:rationally embeddable}
In this section, we derive some restrictions on a symplectic circular spherical divisor embedded in a symplectic rational surface. In particular, we give a complete list of rationally embeddable circular spherical divisors up to toric equivalence in Proposition \ref{prop:list}.

By smoothing a symplectic circular spherical divisor $ D $ to a symplectic torus $ T $ with $ [T]=[D] $ and applying Theorem 6.10 in \cite{OhOn03-simple-elliptic}, we get the following upper bounded on $ [D]^2 $.
\begin{lemma}[\cite{OhOn03-simple-elliptic}]\label{lemma:constraint on D^2}
	For a symplectic circular spherical divisor $D$ in a symplectic rational surface $(X, \omega)$ we have that $[D]^2\leq 9$. 
\end{lemma}

This upper bound becomes an upper bound on $ r(D) $ when $ s_i\ge -1 $ for all $ i $. Combined with the homological classification in Lemma \ref{lemma:topological cyclic}, we can determine exactly when such $D$ is rationally embeddable.
\begin{lemma}\label{lemma:s_i>=-1 embeddable}
	Let $ D $ be a circular spherical divisor with $ s_i\ge -1 $ for all $ i $. Then it is rationally embeddable if and only if it is toric equivalent to $ (1,1,1) $ or $ (-1,k), -1\le k\le 5 $ or $ (1,k), 0\le  k\le 4 $ or one in the family $ \mc{F}(2,2) $ of Lemma \ref{lemma:topological cyclic} (6).
\end{lemma}
\begin{proof}
	
	Let $ D=(s_1,\dots,s_r) $ be a circular spherical divisor. In the case $ r\ge 3 $, we might assume $ D $ is toric minimal and $ s_i\ge 0 $. Suppose $ D $ is embedded in symplectic rational surface $ (X,\omega) $, then $ 9\ge [D]^2=\sum_{i=1}^r (s_i+2)\ge 2r $, thus $ r\le 4 $. By Lemma \ref{lemma:topological cyclic}, we have that $ D $ can be one of the following: $ (0,0,0,0), (1,1,1), (1,1,0),(k,0,0),0\le k\le 2 $. Note that $ (0,0,0,0) $ is toric equivalent to $ (1,0,-1,0) $ using the balancing move in Example \ref{eg: balancing self-intersection by $0$-sphere} and thus toric equivalent to $ (1,1,1) $. Similarly, combining the balancing move and toric blow-down, we have that $ (1,1,0) $ is toric equivalent to $ (4,1) $ and $ (k,0,0) $ to $ (k+2,1) $. It's easy to check they are indeed rationally embeddable as they can be realized as blow-ups of the divisors in Theorem \ref{thm:minimal model}.
	
	In the case $ r=2 $, we cannot assume $ D $ to be toric minimal. Again by Lemma \ref{lemma:topological cyclic}, we have that $ D $ is $ (k,-1), -1\le k $ or one of $ (4, 1), (4, 0), (3, 1), (3,0),  (2, 2), (2, 1), (2, 0), 
	(1, 1),\\ (1, 0), (0, 0) $. In the first case, we have $ 8\ge D^2=k+3  $ because $ X $ contains at least one exceptional class. So $ k\le 5 $. They are all rationally embeddable as blow-ups of the divisors in Theorem \ref{thm:minimal model}.
\end{proof}
Given two circular spherical divisors $ D,D' $ of length $ l $, 
we say  $ D $ is \textbf{blown-up} if $ D$ can be obtained from non-toric blowing up $ D_0 $, for some toric blow-up $ D_0 $ of $ (1,1,1) $. Note that $ D=(1,1-p_1,-p_2,\dots,-p_{l-1},1-p_l) $ being blown-up in our definition is equivalent to the dual cycle of $ (-p_1,\dots,-p_l) $ being embeddable in the sense of \cite{GoLi14}. In Theorem 3.1 (iii) of \cite{GoLi14}, it is proved that minimal symplectic fillings of the boundary contact torus bundles of such divisors always have vanishing first Chern class. Their proof translates to the following lemma in our setting.
\begin{lemma}[\cite{GoLi14}]\label{lemma:hyperbolic sequence embeddable}
	If $ D=(1, -p_1+1, -p_2, ...., -p_{l-1}, -p_l+1)$ with $p_i\geq 2$ and $l\geq 2$, then it is rationally embeddable, anti-canonical and rigid if and only if it is blown-up.
\end{lemma}
\begin{proof}
		If $ D $ is of such form and is blown-up, then $ D $ is realized by blowing up three lines in general position inside $ \CC\PP^2 $ and is rationally embeddable (Lemma 2.4 of \cite{GoLi14}). The converse was actually contained in the proof of Theorem 3.1 (iii) of \cite{GoLi14} but was not explicitly written in their theorem. We recall their arguments here for readers' convenience. 
		Let $ D $ be embedded in a symplectic rational surface $ (X,\omega ) $ with $ D=(1,1-p_1,\dots,-p_{l-1},1-p_l) $, we blow up $ D $ to $ D' $ in $ X'=\CC\PP^2\# M\overline{\CC\PP}^2 $ with $ M\ge 1 $ and $ D'=(1,1-p_1,-p_2,\dots,-p_{l-1}-1,-1,-p_l) $. We choose an $ \omega  $-tame almost complex structure $ J $ on $ X $ that makes all the symplectic spheres in $ D' $ $ J $-holomorphic.
		Let $ S' $ be the irreducible component in $ D' $ with $ [S']^2=-p_l $ and let $ \tilde{D}'=D'-S' $ be the symplectic string with intersection sequence $ (1,1-p_1,-p_2,\dots,-p_{l-1}-1,-1) $. By Theorem 4.2 of \cite{Lis08-lens}, there is a sequence of symplectic blowdowns of $ X' $ to $ \CC\PP^2 $ such that $ \tilde{D}' $ blows down to the union of two lines $ L\cup L'\subset \CC\PP^2 $. At each blow-down, the almost complex structure $ J $ descends. Since the complement of $ D' $ is minimal, the exceptional divisors we blow down either intersect $ D' $ once or contained in $ D' $.
		During this process, $ S' $ blows down to a smoothly embedded symplectic sphere intersecting both $ L $ and $ L' $ exactly once, hence $ S' $ blows down to a line. So $ D $ is a blow-up of $ (1,1,1) $ corresponding to three lines in $ \CC\PP^2 $, which exactly means $ D $ is blown-up.
\end{proof}
\begin{lemma}\label{lemma:k+p=5}
	$(5+p, -p)$ with $ p\ge -2 $ is not rationally embeddable if $ p\neq -1 $.
\end{lemma}
\begin{proof}
	Suppose $ D=(5+p,-p) $ is symplectically embedded in a symplectic rational surface $ (X,\omega) $. Observe that if $ p\ge 0 $, $ X $ cannot be $ \CC\PP^2 $ because there is no second homology class with negative self-intersection in $ \CC\PP^2 $. Also, $ (3,2) $ cannot be embedded into $ \CC\PP^2 $ either because there is no second homology class with self-intersection $ 2 $ or $ 3 $. So we could assume $ X=\CC\PP^2\# l\overline{\CC\PP}^2 $ for some $ l\ge 1 $ or $ X=S^2\times S^2 $.
	
%
%
	Now we use the standard form of sphere classes with positive square to show such configuration of symplectic spheres cannot be embedded in a symplectic rational surface.
	
	\textbf{If $ C_1 $ is an odd sphere} with $[C_1]^2=2x+1\geq 3$, 
	by \cite{Mc90-structure}, there is a symplectomorphism of $ X $ to a blow-up of a Hirzebruch surface which identifies $ C_1 $ with the section. Then in terms of the standard basis $ H_2(X;\ZZ)=\ZZ\{ h,e_1,\dots,e_l \} $, we have 
	\begin{align*}
	&[C_1]=  (x+1)h -x e_1,\\ &[C_2]=ah-be_1 - \sum_{i= 2}^l b_i e_i,\\
	&[C_2]^2=a^2-b^2-\sum_{i=2}^l b_i^2=5-( 2x+1)=-2x+4,\\
	&[C_1]\cdot [C_2]=(x+1)a-xb=2.
	\end{align*}
	The equation with $ [C_2]^2$ implies that  $a^2-b^2\geq -2x+4$. Any solution to $ (x+1)a-xb=2 $ is of the form $(a, b)=(ux+2, u(x+1)+2)$  for an integer $u$. 
	We have that 
	\begin{align*}
	a^2-b^2&=(ux+2)^2-(u(x+1)+2)^2\\&=-u(2ux+u+4)=(-2x-1)(u^2+\frac{4}{2x+1}u)=:F(u).
	\end{align*}
	First we check that $ u\neq 0 $. If $ u=0 $, then $ (a,b)=(2,2) $. The adjunction formula for $g(C_2)=0$ is  $(a-1)(a-2) -b(b-1)- \sum b_i(b_i-1)=0$. Since $y(y-1)\geq 0$ for any integer $y$, we have 
	$$(a-1)(a-2)\ge b(b-1).$$
	Clearly $(a,b)=(2,2)$ violates this inequality.
	
	Since  $0<\frac{4}{2x+1}<2$ the values of $(u^2+\frac{4}{2x+1}u)$ for an integer $u$ are smallest when $u=0$ or $-1$. We have $F(u)\leq F(-1)=-2x+3 < -2x+4$ for $u\leq -1$ and $F(u)\leq  F(1)=-2x-5< -2x+4$ for $u\geq 1$. But this contradicts $ a^2-b^2\ge -2x+4 $.

	\textbf{If $ C_1 $ is an even sphere} with $[C_1]^2=2x\geq 6$, 
	by \cite{Mc90-structure}, there is a symplectomorphism of $ X $ to a blow-up of a Hirzebruch surface which identifies $ C_1 $ with the section. Then in terms of the standard basis $ H_2(X;\ZZ)=\ZZ\{ s,f,e_1,\dots,e_l \}  $, we have 
	\begin{align*}
	&[C_1]=  s+ x f,\\
	&[C_2]=as+ bf - \sum_{j=1}^l b_j E_j,\\
	&[C_2]^2=2ab-\sum_{j=1}^l b_j^2=5-2x,\\
	&[C_1]\cdot [C_2]=ax+b=2.
	\end{align*}
	The equation with $ [C_2]^2 $ implies that $ 2ab\geq -2x+5$. Any solution to $ax+b=2$ is of the form $(u, -ux+2)$ for an integer $u$. Similarly, we have that \begin{align*}
	2ab&=2u(-ux+2)\\ &=-2x u^2+4u=-2x(u^2+\frac{4}{2x}u)=:G(u).
	\end{align*}
	Again we check that $ u\neq 0 $. If $ u=0 $, then $ (a,b)=(0,2) $. The adjunction formula for $g(C_2)=0$ is  $2(a-1)(b-1)- \sum b_i(b_i-1)=0$
	Since $y(y-1)\geq 0$ for any integer $y$, we have 
	$$(a-1)(b-1)\ge 0.$$
	Clearly $(a,b)=(0,2)$ violates this inequality. 
	
	Since $0<\frac{4}{2x}<2$, we have $G(u)\leq G(-1)=-2x+4< -2x+5$ for $u\leq -1$ and $G(u)\leq G(1)=-2x-4< -2x+5$ for $ u\ge 1 $. But this contradicts $ 2ab\ge -2x+5 $.
	
\end{proof}
\begin{prop}\label{prop:list}
	Any rationally embeddable circular spherical divisor with $ b^+\ge 1 $ is toric equivalent to one in the list:
	\begin{enumerate}[\indent$  (1) $]
		\item $ (1,1,p) $ with $ p\le 1 $.
		\item $ (1,p) $ with $ p\le 4 $.
		\item $ (0,p) $ with $ p\le 4 $.
		\item $ (1,1-p_1,-p_2,\dots,-p_{l-1},1-p_l) $ with $ p_i\ge 2 $, $ l\ge 2 $.
		\item $ s_i\ge -1 $ for all $ i $.
		\item $ (-1,-2) $, $ (-1,-3) $
	\end{enumerate} 
\end{prop}
\begin{proof}
	The proof is a case-by-case analysis based on the length of the divisor. Since a rationally embeddable circular spherical divisor must have $ b^+\le 1 $, Lemma \ref{lemma:topological cyclic} applies here.
	
	\noindent\textbf{Case 1}: $r(D)=2$.
	\begin{itemize}
		\item When $r^{\geq 0}(D)=2$, the divisors are listed in (6) of Lemma \ref{lemma:topological cyclic}. They all belong to (5) of the list.
		\item When  $r^{\geq 0}(D)=1$, they are of the forms $(k, p)$ with $ k\ge 0 $, $ p<0 $ and $k+p\leq 5$ by Lemma \ref{lemma:constraint on D^2}. If $ k+p=5 $, then by Lemma \ref{lemma:k+p=5} the only rationally embeddable one is $ (1,4) $ belonging to (2). Now suppose $ k+p\le 4 $.
		If $ k\le 1 $, the divisor $ (k,p) $ belongs to (2) or (3) of the list. If $ k\ge 2 $, we can toric blow up the pair $ (C_1,C_2) $, and if necessary, apply successive toric blow-ups to the pairs of the proper transform of $ C_1 $ and the exceptional spheres to get $ \bar D $ with $ \bar s_1=1, \bar s_2 = -1, \bar s_i\le -2 $ for $ i\ge 3 $. Then $ \bar D $ belongs to (4) of the list.
		\item When $r^{\geq 0}(D)=0$, by (8) of Lemma \ref{lemma:topological cyclic}, there are only 3 circular spherical divisors with $b^+(Q_D)=1$: $(-1, -1), (-1, -2)$ and  $(-1, -3)$. $ (-1,-1) $ belongs to (5) and $ (-1,-2),(-1,-3) $ belong to (6) of the list.
	\end{itemize}
	
	\noindent\textbf{Case 2}: $r(D)=3$ and $ D $ is toric minimal (if not,  reduce to the $r=2$ case). 
	\begin{itemize}
		\item When  $r^{\geq 0}(D)=3$, it belongs to (5) of the list.
		\item When $r^{\geq 0}(D)=2$, by (5) of Lemma \ref{lemma:topological cyclic}, we can assume that $D=(1, 1, p\leq 1)$, or $D=(s_1\geq 0, 0, s_3)$ with $ s_1+s_3\leq 2$. The former case is already in the list. In the latter case, we can apply the balancing move as in Example \ref{eg: balancing self-intersection by $0$-sphere} based at $ C_2 $ to decrease $ s_1 $ to $\bar s_1=0 $ and increase $ s_3 $ to $ \bar s_3=s_3+s_1 $. Denote the new divisor by $ \bar D $. By toric blowing up the pair $ (\bar C_1,\bar C_2) $ and contracting the proper transform of $ \bar C_1,\bar C_2 $, we can reduce the length of the divisor to 2. So $ D $ is toric equivalent to one in the list.
		\item When $r^{\geq 0}(D)=1$, 
		if $s_1=1$, then it belongs to (4) of the list.
		If  $s_1\geq 2$, toric blow up the pair $C_1, C_2$, and if necessary, 
		apply successive toric blow-ups to the pairs of  the proper transforms  of $C_1$ and the exceptional spheres to get $\bar D$ with $\bar s_1=1, \bar s_2=-1, \bar s_i\leq -2$  for $i\geq 3$ (so $\bar D$ is not  toric minimal), so $\bar D$ belongs to (4) of the list.
		
		If $s_1=0$, apply the balancing move based at $C_1$ to increase  $s_2$ to $\bar s_2=0$ (while decreasing $s_r$ to $\bar s_r=s_r + s_2$). Notice that  $r^{\geq 0}(\bar D)=2, \bar s_1=\bar s_2=0$, and $\bar D$ toric minimal. We treated this case above. 
	\end{itemize}
	
	\noindent\textbf{Case 3}: $r(D)=4$ and $D$ is toric minimal. 
	\begin{itemize}
		\item When $r^{\geq 0}(D)=4$,  it is $(0, 0, 0, 0)$
		by (3) of Lemma \ref{lem: non-negative components}.	Note that any divisor of form $ (0,0,0,p) $ is toric equivalent to the divisor $ (1,1,p+1) $ by toric blowing up $ (C_1,C_2) $ and contracting the proper transforms of $ C_1,C_2 $.
		\item When $r^{\geq 0}(D)=3$, by (2) of Lemma \ref{lemma:topological cyclic}, $(s)=(k, 0, p, 0)$ with $k\geq 0, p<-k$. Using the balancing move based at $ C_2 $, $ (k,0,p,0) $ is toric equivalent to $ (0,0,k+p,0) $ and thus toric equivalent to $ (1,1,k+p+1) $.
		\item When $r^{\geq 0}(D)= 2$, by (3) of Lemma \ref{lemma:topological cyclic} we have $(0, p_1, 0, p_2), p_i< 0$ or $(k\geq 0, 0, p_1, p_2),p_i<0, k+p_1+p_2\leq 0$. 
		
		For the case $(k, 0, p_1, p_2)$ we apply the balancing move based at $ C_2 $ to transform $ D $ to $\bar D$ with  $\bar s_1=\bar s_2=0, \bar s_3=s_3+s_1=k+p_1, \bar s_4=s_4=p_2\leq -2$. By blowing up the pair $ (\bar C_1,\bar C_2) $ and contracting the proper transforms of $ \bar C_1,\bar C_2 $, we get a divisor $ \bar D' $ of length 3.
		
		For the case  $(0, p_1, 0, p_2), p_i<0$, using the balancing move based at $ C_3 $, it is toric equivalent to $ (0,0,0,p_1+p_2) $ and thus toric equivalent to $ (1,1,p_1+p_2+1) $.
		\item When $r^{\geq 0}(D)= 1$,   if $s_1=1$ and $ s_i\le -2 $ for $ i\ge 2 $, then it belongs to (4) of the list.
		
		If  $s_1\geq 2$, toric blow up the pair $C_1, C_2$, and if necessary, 
		apply successive toric blow-ups to the pairs of  the proper transforms  of $C_1$ and the exceptional spheres to get $\bar D$ with $\bar s_1=1, \bar s_2=-1, \bar s_i\leq -2$  for $i\geq 3$, so $\bar D$ belongs to (4) of the list.
		
		If $s_1=0$, apply the balancing move based at $C_1$ to increase  $s_2$ to $\bar s_2=0$ (while decreasing $s_r$ to $\bar s_r-s_2$). Notice that  $r^{\geq 0}(\bar D)=2, \bar s_1=\bar s_2=0$, and $\bar D$ toric minimal. We have treated this case above. 
	\end{itemize}
	
	\noindent\textbf{Case 4}: $ r(D)\geq 5$ and $ D $ is toric minimal.
	
	This is proved by induction. 
	Suppose we have proved the case $r(D)\leq n$ with some $n\geq 4$, where $D$ is not assumed to be toric minimal. The case $ r(D)=n+1 $ and $ D $ not toric minimal follows directly from induction hypothesis by toric blow-down. We will verify the case where $r(D)=n+1$ and $D$ toric minimal. 
	
	We have $r^{\geq 0}(D)\geq 1$ by (1) of Lemma \ref{lem: not negative semi-definite => at least one non-negative} and we may assume that $s_1\geq 0$. By (1) of Lemma \ref{lemma:topological cyclic}, $r^{\geq 0}(D)\leq 2$. 
	\begin{itemize}
		\item When $r^{\geq 0}(D)=2$, by (1) of Lemma \ref{lemma:topological cyclic}, we can assume that  $s_1\geq s_2=0$. Apply the balancing move based at $C_2$ to transform to $\bar D$ with  $\bar s_1=\bar s_2=0, \bar s_3=s_3+s_1, \bar s_i=s_i\leq -2$ for $ 4\le i\le n+1 $. Toric blow up the pair $(\bar C_1, \bar C_2)$ and then contract the proper transforms of $\bar C_1$ and $ \bar C_2$ to get $\bar D'$ with $\bar s_1'=1, \bar s_2'=\bar s_3+1, \bar s_{n}'=\bar s_{n+1}+1$. Since $ r(\bar D')=r(D)-1=n$, the induction hypothesis applies. 
		%
		%
		\item Suppose $r^{\geq 0}(D)=1$ and we may assume $ s_1\ge 0 $. If $s_1=1$, then it belongs to (4) of the list.
		If  $s_1\geq 2$, toric blow up the pair $C_1, C_2$, and if necessary, 
		apply successive toric blow-ups to the pairs of  the proper transforms  of $C_1$ and the exceptional spheres to get $\bar D$ with $\bar s_1=1, \bar s_2=-1, \bar s_i\leq -2$  for $i\geq 3$, which belongs to (4) of the list.
		
		If $s_1=0$, apply the balancing move based at $C_1$ to increase  $s_2$ to $\bar s_2=0$ (while decreasing $s_{n+1}$ to $\bar s_{n+1}=s_{n+1}+s_2< -1$ as $ s_2,s_{n+1}< -1 $). Notice that  $r^{\geq 0}(\bar D)=2, \bar s_1=\bar s_2=0$, and $\bar D$ is toric minimal. We have treated this case above.
	\end{itemize}
\end{proof}
\begin{proof}[Proof of Theorem \ref{thm:list}]
	The theorem is basically Proposition \ref{prop:list} with the exception of a few cases. We still need to show that the list in Proposition \ref{prop:list} is toric equivalent to the list in the theorem and every circular spherical divisor in the theorem can indeed be symplectically embedded in a rational manifold. This is done by Lemma \ref{lemma:s_i>=-1 embeddable}, Lemma \ref{lemma:hyperbolic sequence embeddable} and simply observing the following toric equivalences (denoted by $ \sim $): \begin{align*}
	(-1,k)&\sim (1,-1,-2,\dots,-2) \text{ with $ k-1 $ number of $ -2 $}, k\ge 2\\
	(1,4)&\sim (3,-1,0) \sim (1,1,0)\\
	(2,2)&\sim (1,1,-1)
	\end{align*}
	As in Section \ref{subsection:torus bundle}, each divisor $D=(s_1,\dots,s_r)$, its negative boundary $-Y_D$ is a torus bundle with monodromy $A(-s_1,\dots,-s_r)^{-1}$.
	So their monodromy can be easily calculated. 
	For family (4) where $D=(1,p)$ or $(1,1-p_1,-p_2,\dots,-p_{l-1},1-p_l)$, by smoothly blowing down the $(+1)$ component, it has the same boundary $Y_D$ as $(-p_1,\dots,-p_l)$ with $l\ge 1$, $p_i\ge 2$ and at least some $p_j\ge 3$, which is negative hyperbolic by Theorem 6.1 of \cite{Ne81-calculus}. Note that if $\tr(A)<-2$, then $\tr(A^{-1})<-2$. So $-Y_D$ is also negative hyperbolic.

	Finally we show all 3-manifolds in the list are distinct. For families (1)(2)(3), the monodromies are all distinct. 
	For family (4), as we mentioned above, the family of torus bundles $Y_D$ are the same as the family of torus bundles obtained as boundaries of $(-p_1,\dots,-p_l)$ with $l\ge 1$. Note that by Proposition 6.3 of \cite{Ne81-calculus}, the divisors 
	are in bijection with all conjugacy classes of elements $ A\in SL(2,\Z) $ 
	with trace $\ge 3$. Since the classification of conjugacy classes of $A\in SL(2,\ZZ)$ with trace $\le -3$ is equivalent to the classification for trace $\ge 3$ by multiplying $-1$, we conclude that the divisors $(-p_1,\dots,-p_l)$ are in bijection with conjugacy classes 
	of elements $ A\in SL(2,\Z) $ with trace $\le -3$. Hence all torus bundles in family (4) are distinct.
\end{proof}

\subsection{Rationally embeddable $D$ is anti-canonical}

\begin{lemma}\label{lemma:s_i>=-1}
	Suppose $ T $ is an embedded symplectic torus in a symplectic rational surface $ (X,\omega) $ such that $ [T]^2>0 $ and $ [T]\cdot e\neq 0 $ for all exceptional class $ e $, then $ PD([T])=c_1(X,\omega) $. In particular, a rationally embeddable $ D=(s_1,\dots,s_r) $ with $ s_i\ge -1 $ for all $ i $ is anti-canonical and rigid.
\end{lemma}
\begin{proof}
	By Lemma \ref{lemma:maximal=minimal complement}, the complement of $ T $ is minimal.
	By the discussion in Section 2 of \cite{OhOn03-simple-elliptic}, for $ [T]^2=k $ with $ k> 0 $, the minimal complement of $ T $ gives a minimal symplectic filling for simple elliptic singularity of degree $ k $. Then by Theorem 7.3 (2) in \cite{OhOn03-simple-elliptic}, $ PD([T])=c_1(X,\omega) $.

	Now let $ D $ be symplectically embedded in a symplectic rational surface $ (X,\omega ) $ such that its complement is minimal. We could symplectically smooth $ D $ to a symplectic torus $ T $ with $ [T]^2=[D]^2=\sum (s_i+2)>0 $. Suppose there exists exceptional class $ e $ such that $ [T]\cdot e =0$. Note that if $ s_i\ge 0 $, we have $ [C_i]\cdot e\ge 0 $. If $ s_i=-1 $, then $ [C_i] $ is an exceptional class. Note that $ [C_i]\neq e $ because $ [C_i]\cdot [D]=1 $. Then we have $ [C_i]\cdot e \ge 0 $ since both are stable classes and have positivity of intersection. So $ e $ is $ D- $good and by Lemma \ref{lem: existence of nice J-sphere} there is an almost complex structure $ J $ such that $ D $ is $ J- $holomorphic and $ e $ has an embedded $ J- $holomorphic representative $ E $. Since $ [D]\cdot e=[T]\cdot e =0$, $ D $ and $ E $ are disjoint, contradicting the minimality of the complement of $ D $.
\end{proof}

\begin{prop}\label{prop:anti-canonical}
	Any rationally embeddable circular spherical divisor with $b^+=1$ is anti-canonical.	
\end{prop}
\begin{proof}
	It suffices to show that any rationally embeddable circular spherical divisor listed in Proposition \ref{prop:list} is anti-canonical.
	
	(4) is anti-canonical by Lemma \ref{lemma:hyperbolic sequence embeddable} and (5) is anti-canonical by Lemma \ref{lemma:s_i>=-1}.
	
	(1),(2) and (3) are realized respectively as non-toric blow-ups of divisors (B3),(B2) and (D2) in Theorem \ref{thm:minimal model}.
	
	(6) is also realized as non-toric blow-ups of divisors (B2) or (C2) or (D2) in Theorem \ref{thm:minimal model}.
\end{proof}

\subsection{Anti-canonical $D$ is rigid}\label{subsection:anti-canonical = rigid}

\begin{lemma}\label{lemma:(-1,-3) rigid}
	$ (-1,-3) $ and $ (-1,-2) $ are rigid.
\end{lemma}
\begin{proof}
	By Lemma \ref{lemma:rational-embed}, we may suppose $ D=(-1,-3) $ is a symplectic circular spherical divisor in a symplectic rational surface $ X $. Denote the components of $ D $ by $ A,B $, where $ [A]^2=-1,[B]^2=-3 $. Smooth $ D $ to get a symplectic torus $ T $ with $ [T]^2=0 $. Note that we must have $ K_X^2 \le 0 $. Otherwise the subspace in $ H_2(X;\ZZ ) $ spanned by $ 3[A]+2[B] $ and $ -K_{X} $ has intersection form $ \begin{pmatrix}
	3 & 1\\ 1 & (-K_{X})^2
	\end{pmatrix} $ and is positive definite, contradicting $ b^+(X)=1 $. Here $ (3[A]+2[B])\cdot (-K_X)=1 $ because $ (-K_X)\cdot [A]=[A]^2+2=1 $ and $ (-K_X)\cdot [B]=[B]^2+2=-1 $ by adjunction formula. So we have that $ X $ must be $ \CC\PP^2\#l\overline{\CC\PP}^2 $ with $ l\ge 9 $.
	
	When $ l=9 $, we have that $ [T]^2=K_{X}^2=0 $ and $ K_{X}\cdot [T]=0 $. By light cone lemma, we have $ [T] $ is proportional to $ -K_{X} $. Since $ [T]\cdot [A]=([A]+[B])\cdot [A]=1=-K_{X}\cdot [A] $, we actually have $ [D]=[T]=-K_{X} $, i.e. $ D $ is anti-canonical.
	
	When $ l\ge 10 $, suppose $ [T]\cdot e\neq 0 $ for all exceptional class $ e $. By Proposition 3.14 of \cite{LiZh11-relative}, we must have $ (K_{X}+[T])^2\ge 0 $. However, $ (K_{X}+[T])^2=K_{X}^2+2K_{X}\cdot [T] + [T]^2=K_{X}^2=9-l<0 $. So there exists an exceptional class orthogonal to $ [T] $. By Theorem 3.21 of \cite{LiZh11-relative}, exceptional classes orthogonal to $ [T] $ are pairwise orthogonal and above discussion implies that there are at least $ l-9 $ such exceptional classes. Denote them by $ e_1,\dots ,e_{l-9} $.
	
	Now consider the blowup class $ \tilde{K}=K_{X}-e_1-e_2 -\dots - e_{l-9} $. 
	Then we have \[
	[T]^2=0, \quad (-\tilde{K})^2=0, \quad \text{ and }[T]\cdot (-\tilde{K})=0.
	\]
	By the light cone lemma, we have $ [T] $ and $ -\tilde{K} $ are proportional. Pairing both with $ [A] $, we have $ [T]\cdot [A]=1=(-\tilde{K})\cdot [A] $. Therefore we conclude that $ [T]=-\tilde{K}=-K_{X_l}+e_1+\dots + e_{l-9} $.
	
	Now we have \begin{align*}
	1=([A]+[B])\cdot [A]=[T]\cdot [A]=(-K_{X_l}+e_1+\dots + e_{l-9})\cdot [A]=1+\sum e_i\cdot [A].
	\end{align*}
	Since both $ e_i,[A] $ are stable classes, we must have $ e_i\cdot [A]\ge 0 $ and thus $ e_i\cdot [A]=0 $ for all $ i $. This implies $ e_i\cdot [B]=e_i\cdot ([T]-[A])=0 $. So each $ e_i $ is $ D- $good and has an embedded symplectic representative in the complement of $ D $ by Lemma \ref{lem: existence of nice J-sphere}.
	
	So $ D $ has minimal complement only if $ l=9 $, where $ D $ is anti-canonical.

	The proof for $(-1,-2)$ follows the exact same line as the proof for $ (-1,-3) $. Suppose $ D=(-1,-2) $ is a symplectic circular spherical divisor embedded in a symplectic rational surface $ (X,\omega) $ and $ T $ is the symplectic torus we get from smoothing $ D $. Note that we must have $ K_X^2\le 1 $. Otherwise the subspace spanned by $ [T] $ and $ -K_{X} $ has intersection form $ \begin{pmatrix}
		1 & 1\\ 1 & (-K_X)^2
		\end{pmatrix} $ and is positive definite, contradicting $ b^+(X_l)=1 $. So $ X $ must be $ \CC\PP^2\#l\overline{\CC\PP}^2 $ with $ l\ge 8 $. We could show that $ D $ has minimal complement only if $ l=8 $. When $ D $ has minimal complement, the symplectic torus $ T $ we get from smoothing $ D $ pairs nontrivially with any exceptional class and $ [T]^2=1>0 $. By Lemma \ref{lemma:s_i>=-1}, $ T $ represents $ c_1(X,\omega ) $, i.e. $ PD([D])=c_1(X,\omega) $.
\end{proof}
\begin{prop}\label{prop:rigid}
	Anti-canonical circular spherical divisors with $ b^+=1 $ are rigid.
\end{prop}
\begin{proof} 
	It suffices to show anti-canonical circular spherical divisors of forms listed in Proposition \ref{prop:list}	are rigid.
	
	(6) is rigid by 
Lemma \ref{lemma:(-1,-3) rigid}.
	
	(5) is rigid by Lemma \ref{lemma:s_i>=-1}.
	
	(4) is rigid by Lemma \ref{lemma:hyperbolic sequence embeddable}.
	
	(3) follows from Theorem 3.5 in \cite{GoLi14}. 
	
	The case of $p\leq -2$ in (2) follows from Theorem 3.1 in \cite{GoLi14}, while the 
	remaining divisors in (2) have $s_i\ge -1,\forall i$ and follows from (5). 
	
	Now we want to show the divisor $ (1,1,p) $ is rigid. By \cite{Mc90-structure}, $ (X,\omega ) $ is rational and it must be $ \CC\PP^2\# l\overline{\CC\PP}^2 $. By (4) and (5) of Lemma \ref{lemma:topological cyclic} together with \cite{Mc90-structure}, we may assume that $ [C_1]=[C_2]=h $ and $ [C_3]=h-\sum a_ie_i $, where $ \{h,e_1,\dots,e_l \} $ is a basis of $ H_2(X;\ZZ) $. Adjunction formula for $ C_3 $ says $ \sum (a_i^2-a_i)=0 $, which implies each $ a_i $ is $ 0 $ or $ 1 $. If $ a_k=0 $ for some $ k $, then $ [C_j]\cdot e_k=0 $ for $ j=1,2,3 $. Then $ e_k $ is $ D- $good and there is a symplectic sphere representing $ e_k $ in the complement of $ D $, contradicting the minimality.
\end{proof}

\section{Embeddability and rigidity when $b^+(D)=0$}
This section is devoted to the study of embeddability and rigidity of negative semi-definite circular spherical divisors. The main result of this section is that all such divisors are symplectically embeddable and not rigid. 
We also classify anti-canonical strictly negative semi-definite divisors up to toric equivalence and show that the condition $b^+\ge 1$ in Theorem \ref{thm:embeddable=rigid} is sharp.
\subsection{Strictly negative semi-definite cases}\label{section:non-example}
In this subsection, we study strictly negative semi-definite circular spherical divisors, which means they are negative semi-definite but not negative definite. Recall that by Lemma \ref{lemma:toric eq preserves}, being symplectically embeddable, rationally embeddable, anti-canonical and rigid are all preserved by toric equivalence, so it suffices to consider toric minimal divisors. By Lemma \ref{lem: not negative semi-definite => at least one non-negative} any strictly negative semi-definite divisor is toric equivalent to $D_1=(-1,-4)$ or \[ D_n=(\underbrace{-2,\dots,-2}_{n}), n\ge 2. \]

We first introduce another convenient operation on circular spherical divisors.
\begin{definition}
	Let $D=(s_1,\dots,s_r)$ be a circular spherical divisor. Then $D'=(s_1',\dots,s_{r-1}')$ is a smoothing of $D$ at the intersection $(s_i,s_{i+1})$ if $s_j'=s_j$ for $j<i$, $s_j'=s_{j+1}$ for $j>i$, and $s_i'=s_i+s_{i+1}+2$.
\end{definition}
For a symplectic circular spherical divisor $D=\cup C_i \subset (X,\omega)$, we can symplectically smooth the intersection point of $C_i$ and $C_{i+1}$ to get a symplectic embedding of $D'$ into the same symplectic 4-manifold $(X,\omega)$ (\cite{LiUs06-negative-inflation}). Clearly smoothing of a symplectically embeddable/rationally embeddable divisor is still symplectically embeddable/rationally embeddable. Also because the symplectic smoothing does not change the total homology class, i.e. $[D]=[D']\in H_2(X;\ZZ)$, being anti-canonical is also preserved. Note that any smoothing of $D_{n+1}$ is $D_{n}$ for $n\ge 2$, which makes smoothing a convenient reduction tool in the negative semi-definite case.

To find symplectic embeddings of $D_n$, we also need to recall some basic facts of elliptic surfaces (\cite{Barth15-compact}). An elliptic fibration of a complex surface $X$ is a proper holomorphic map $f:X\to S$ to a smooth curve $S$, such that the general fiber is a non-singular elliptic curve. An elliptic surface is a surface admitting an elliptic fibration. 
Following the classification of singular fibers by Kodaira, a singular fiber of $I_n$ type is a cycle of $n$ self-intersection $(-2)$ rational curves for $n\ge 2$, and a rational nodal curve with self-intersection $0$ for $n=1$. 
We restrict ourselves to elliptic fibrations that are relatively minimal (i.e. fibers contain no $ (-1) $-curves), have at least one singular fiber and have a global section.
Such an elliptic fibration $f:X\to S$ has Kodaira dimension $\text{kod}(X)=-\infty$ if and only if $\chi(\mc{O}_X)+2g(S)-2<0$ (Proposition 12.5 in \cite{Barth15-compact}). Here the holomorphic Euler characteristic $ \chi(\mc{O}_X)$ is related to topological Euler characteristic $ \chi(X) $  by $\chi(\mc{O}_X)=\chi(X)/12 $ by (12.5) in \cite{Kod-surface3}. The elliptic surface $X$ has even first Betti number $b_1(X)=2g(S)$ (Corollary 5.2.2 in \cite{CD-Enriques}), and in particular $X$ is K\"ahler by \cite{Miy-kahler}. As a K\"ahler surface, it is not symplectically rational or ruled if and only if it has non-negative Kodaira dimension, since symplectic and holomorphic Kodaira dimensions coincide (\cite{Li06-kod-0}).

\begin{lemma}\label{lemma:D_n symp embeddable}
	Let $D$ be a strictly negative semi-definite circular spherical divisor. Then $D$ is symplectically embeddable. In particular, $D$ symplectically embeds into a symplectic 4-manifold which is not symplectically rational and is thus not rigid.
\end{lemma}
\begin{proof}
	It suffices to consider $D_n$ with $n\ge 1$. 
	For each $n\ge 2$, we can find $m\ge 1$ with $10m-1\ge n$. By Theorem 2.3 of \cite{Shioda-elliptic}, for each $m\ge 1$, there exists a minimal elliptic surface $X$ over $\PP^1$, with a maximal singular fiber of type $I_{10m-1}$, $2m+1$ singular fibers of type $I_1$. This $X$ is an elliptic modular surface by Theorem 2.4 of \cite{Shioda-elliptic} and has a global section by Definition 4.1 of \cite{Shioda-modular}. It has $\chi(\mc{O}_X)=\chi(X)/12=[\dfrac{2m+1}{2}]\ge 2$ by Theorem 2.2 in \cite{Shioda-elliptic} and thus non-negative Kodaira dimension as discussed above. Equip $X$ with a K\"ahler form $\omega$. Then $(X,\omega)$ is not symplectically rational or ruled and is symplectically minimal.
	The $I_{10m-1}$ singular fiber gives a symplectic embedding of $D_{10m-1}$. Symplectically smoothing $D_{10m-1}$ gives a symplectically embedding of $D_n$ into $(X,\omega)$. 


	Note that a singular fiber of type $I_1$ in $X$ is an immersed symplectic sphere with one nodal point and self-intersection $0$. Blowing up the nodal point gives a symplectic embedding of $D_1=(-1,-4)$ into $X\# \overline{\CC\PP}^2$ with some blow-up symplectic form. 
\end{proof}


\begin{prop}\label{prop:s_n}
	The circular spherical divisor $D_n$ is anti-canonical if $n\le 9$ and is not anti-canonical if $n\ge 10$. $D_n$ is not rigid for all $n\ge 1$.
\end{prop}
\begin{proof}
	The circular spherical divisor $D_1,D_2$ can be obtained by blowing up from the divisor $(1,4)$, $D_3$ from $(1,1,1)$ and $D_n$ from $(0,0,0,0)$ for $4\le n \le 9$. So for $1\le n\le 9$, each $D_n$ is anti-canonical.

	Recall the charge $ q(D)=12-3r(D)-\sum s_i $. A toric blow-up preserves the charge and a non-toric blow-up increases the charge by $ 1 $. Since every anti-canonical divisor $ D $ in Theorem \ref{thm:minimal model} has $ q(D)\ge 0 $ and every anti-canonical circular spherical divisor $ D' $ is obtained from one such divisor through a sequence of toric and then non-toric blow-ups, we must have $ q(D')\ge 0 $. 
	Now we consider the case when $ n\ge 11 $.
	Since there is no $(-1)$-component in $D_n$, it must be obtained by a non-toric blow-up from a divisor of the form \[ (\underbrace{-2,\dots, -2}_{n-1},-1).\] 
	Toric blow it down to $ (\underbrace{-2,\dots, -2}_{n-3},-1,-1) $. Then toric blowing down either of the $ (-1) $-components gives the same divisor $ (\underbrace{-2,\dots, -2}_{n-4},-1,0) $ up to (anti-)cyclic permutations. We then keep blowing down the unique $ (-1) $-component at each stage to get 
	\[ (\underbrace{-2,\dots, -2}_{n-4},-1,0)\to (\underbrace{-2,\dots, -2}_{n-5},-1,1)\to \dots \to (-2,-1,n-5) \to (-1,n-4). \]
	For $n\ge 10$, $(-1,n-4)$ is not one in Theorem \ref{thm:minimal model}. We conclude that there must be at least two non-toric blow-ups and thus $q(D_n)\ge 2$.
	Note that $ q(D_n)=12-n $. So we conclude that $ D_n $ is not anti-canonical for all $ n\ge 11 $.

	It remains to discuss the case $n=10$. For $D_{n}$ to be anti-canonical, it must be obtained from one in Theorem \ref{thm:minimal model} via toric and non-toric blow-ups. So such divisors from Theorem \ref{thm:minimal model} cannot have any component less than $-2$. Also note that $(1,1,1)$ must blow up to $(1,0,-1,0)$ and $(4,1)$ must blow up to $(3,-1,0)$. Then up to (anti-)cyclic permutation, the possible divisors to blow up from are \begin{enumerate}[$(1)$]
		\item $(p,4-p), -2\le p\le 2$;
		\item $(p,0,2-p), -1\le p\le 1$;
		\item $(p,0,-p,0), 0\le p\le 2$,
	\end{enumerate}
	where the bounds on $ p $ is determined by (anti-)cyclic symmetry and the fact that no component has self-intersection less than $ -2 $.


	For $n=10$, suppose $D_{10}$ is obtained from a divisor in case $(i)$ above for $i=1,2,3$. It requires $9-i$ toric blow-ups to reach the correct length and then $i-1$ non-toric blow-ups to reach the correct charge. Since there are always at least two non-toric blow-up, we can only have case $(3)$. 
	By hand enumeration, we see that there is no $6$-time toric blow-ups of case $(3)$ that keeps all components $ \ge -2 $.
	So they cannot be blown up to $D_{10}$.

	The fact that $D_n$ is not rigid follows from Lemma \ref{lemma:D_n symp embeddable} and Remark \ref{rmk:not rigid}.
\end{proof}

Combining Lemma \ref{lemma:D_n symp embeddable} and Proposition \ref{prop:s_n}, we now get non-examples to Theorem \ref{thm:embeddable=rigid} in the strictly negative semi-definite case.
\begin{example}\label{example:negative semidef}
	The circular spherical divisor $D_n$ is \begin{itemize}
		\item anti-canonical but not rigid, for $n\le 9$,
		\item symplectically embeddable but not anti-canonical, for $n\ge 10$.
	\end{itemize}
\end{example}

It seems generally difficult to obstruct or construct rational embeddings of $D_n$. 
The only obstruction we know is that $D_n$ cannot be embedded in $\CC\PP^2\# k\overline{\CC\PP}^2$ for $k<9$. This is because if $D_n$ is embedded in $X=\CC\PP^2\# k\overline{\CC\PP}^2$, we can smooth it to a symplectic torus $T$ with $T^2=0$, which implies $c_1(X)\cdot T=0$. Then by light cone lemma, we must have $c_1(X)^2\le 0$, implying that $k\ge 9$. Note that all anti-canonical $D_n$ must embed in $\CC\PP^2\# 9 \overline{\CC\PP}^2$.
This tempts us to make the following conjecture.
\begin{conjecture}\label{conj:semi-def}
	$D_n$ is rationally embeddable if and only if $D_n$ is anti-canonical.
\end{conjecture}


\subsection{Negative definite case and cusp singularities}\label{section:neg def}
A cusp singularity is the germ of an isolated, normal surface singularity such that the exceptional divisor of the minimal resolution is a cycle of smooth rational curves  $D$ meeting transversely (\cite{EbWall1985}). Note that  $D$ is a negative definite circular spherical divisor. Conversely when $D$ is negative definite, it arises as the resolution of a cusp singularity by Mumford-Grauert criterion (\cite{Grauert1962},\cite{Mumford1961}).

A cusp singularity is called rational if its minimal resolution $ D $ is realized as the anti-canonical divisor of a rational surface, i.e. $D$ is anti-canonical. Every cusp singularity has a dual cusp singularity with its minimal resolution $\check{D}$ to be the dual cycle of $D$. For a negative definite toric minimal divisor \[
	D=(-a_1,\underbrace{-2,\dots,-2}_{b_1},\dots,-a_k,\underbrace{-2,\dots,-2}_{b_k}),
\]
with $a_i\ge 3$ and $b_i\ge 0$, its dual cycle is explicitly given by\[
	\check{D}=(-b_1-3,\underbrace{-2,\dots,-2}_{a_1-3},\dots,-b_k-3,\underbrace{-2,\dots,-2}_{a_k-3}).
\]
Looijenga proved in \cite{Lo81} that if the cusp singularity with exceptional divisor $ D $ is smoothable, then there exists anti-canonical pair $ (Y,\check{D}) $. He conjectured that the converse is also true, which was proved in \cite{GrHaKe11} and \cite{En}.
So the Looijenga conjecture gives an answer to Question \ref{question} for the negative definite circular spherical divisors in terms of its dual cycle.
\begin{theorem}[Looijenga conjecture \cite{Lo81},\cite{GrHaKe11},\cite{En}]
	A toric minimal negative definite circular spherical divisor $D$ is anti-canonical if and only if the cusp singularity corresponding to its dual cycle $\check{D}$ is smoothable.
\end{theorem}
It is a hard (if not harder) problem to determine whether a cusp singularity is smoothable. 
So it is still interesting to understand when negative definite circular spherical  divisors are symplectically embeddable, rationally embeddable, anti-canonical or rigid.

Proposition \ref{prop:s_n} implies that any toric minimal negative definite $D$ with length $\le 9$ is anti-canonical since they are all non-toric blow-ups of $D_n$ with $n\le 9$. Another useful obstruction to negative definite anti-canonical divisors is that they must have $q(D)\ge 3$ by Lemma 4.3 in \cite{FM83}.



Unlike the $b^+\ge 1$ case, negative definite divisors are always symplectically embeddable. Since $D$ is negative definite, the plumbing $N_D$ can be equipped with a convex symplectic structure such that $N_D$ is a symplectic filling. Then $D$ is symplectically embeddable by capping off the filling with a symplectic cap (\cite{EHcap}, \cite{Ecap}).

More concretely, as in Lemma \ref{lem: not negative semi-definite => at least one non-negative}, any toric minimal negative definite circular spherical divisor $D$ can be obtained from some $D_n$ through non-toric blow-ups.
Since by Lemma \ref{lemma:D_n symp embeddable} $D_n$ admits a symplectic embedding into $(X,\omega)$ which is not symplectic rational, then by blowing up $D$ admits a symplectic embedding into some $(X',\omega')$ which is also not symplectic rational. By Remark \ref{rmk:not rigid}, we get the following result.
\begin{prop}\label{cor:negative definite}
	Every negative definite circular spherical divisor is symplectically embeddable and is not rigid.
\end{prop}


Similar to Example \ref{example:negative semidef}, we can find negative definite spherical divisors that violates Theorem \ref{thm:embeddable=rigid}. These examples show that the condition $ b^+\ge 1 $ in Theorem \ref{thm:embeddable=rigid} is sharp.
\begin{example}\label{example:negative def}
	The negative definite circular spherical divisor \[
D=(\underbrace{-2,\dots,-2}_{n-1},-3),
\]
for $ n\ge 11 $, is not anti-canonical since $ q(D)<3 $. But it is symplectically embeddable by Proposition \ref{cor:negative definite}. Also divisor $D=(-2,-n)$ for $n\ge 3$ is negative definite and anti-canonical, but not rigid. 
\end{example}

Similar to the strictly negative semi-definite case, rational embeddability of negative definite divisors seems difficult to study. The only obstruction we know is that they do not embed in $\CC\PP^2\# k\overline{\CC\PP}^2$ with $k< 10$. If $D$ embedds in $X=\CC\PP^2\# k\overline{\CC\PP}^2$, we get a symplectic torus $T$ in $X$ with $T^2<0$ by smoothing $D$. The torus can be made to be $J$-holomorphic by choosing a suitable tame almost complex structure $J$. But such torus doesn't exist in $\CC\PP^2\# k\overline{\CC\PP}^2$ for $k\le 9$ by Proposition 4.2 in \cite{Zhang-curve}.
Note that negative definite anti-canonical $D$ must also embed in $\CC\PP^2\# k\overline{\CC\PP}^2$ for $k\ge 10$. Then we can ask if Conjecture \ref{conj:semi-def} holds for negative definite divisors.
\begin{question}
	Let $D$ be a negative definte circular spherical divisor. If $D$ is  rationally embeddable, is it anti-canonical?
\end{question}

	


\section{Contact aspects}
This section is devoted to the study of contact topology related to circular spherical divisors. Section \ref{section:divisor nbhd} reviews the relation between topological divisors and contact structures, and shows that the induced contact structure only depends on the toric equivalence class of the divisor. 
Then \ref{section:trichotomy} explains that the trichotomy of circular spherical divisors into $b^+>0$, negative definite and strictly negative semi-definite cases is related to the convexity of the divisor. 
In Section \ref{section:fillability} we apply Theorem \ref{thm:embeddable=rigid} and Theorem \ref{thm:list} to prove Theorem \ref{thm:embeddable=fillable}. 
Finally in Section \ref{section:convex Stein} we compute Betti number bounds for Stein fillings when $ D $ is negative definite and raise a conjecture related to Looijenga conjecture.

\subsection{Divisor neighborhood and contact structure}\label{section:divisor nbhd}
In this section, we review some results about the convexity of divisor neighborhoods and the induced contact structure on the boundary. 

Let $ D $ be a symplectic divisor in symplectic 4-manifold $ (W,\omega ) $ (not necessarily closed). A closed regular neighborhood of $ D $ is called a plumbing of $ D $. A plumbing $ N_D $ of $ D $ is called a \textbf{concave/convex plumbing} if it is a strong symplectic cap/filling of its boundary and such $ D $ is called a {\bf concave/convex} symplectic divisor. A concave plumbing is also called a \textbf{divisor cap} of its boundary. Let $ Q_D $ be the intersection matrix of $ D $ and $ a=([C_i]\cdot [\omega])\in (\RR_+)^r $ be the area vector of $ D $. A symplectic divisor $ D $ is said to satisfy the positive (resp. negative) \textbf{GS criterion} if there exists $ z\in (\RR_+)^r $ (resp. $ (\RR_{\le 0})^r $) such that $ Q_D z=a $. 

The GS criterion provides a way to tell when the divisor neighborhood is convex or concave.
\begin{theorem}[\cite{LiMa14-divisorcap}]\label{thm:divisorcap}
	Let $ D\subset (W,\omega ) $ be an $ \omega $-orthogonal symplectic divisor. Then $ D $ has a concave (resp. convex) plumbing if $ (D,\omega ) $ satisfies the positive (resp. negative) GS criterion.
\end{theorem}

Note that a symplectic divisor can always be made $ \omega $-orthogonal by a local perturbation (\cite{Gom95-fibersum}). 
A necessary condition for $D$ to have concave or convex plumbing is $\omega$ being exact on the boundary $ Y_D $. To determine the exactness of $ \omega|_{Y_D} $, it suffices to check the following local criterion.

\begin{lemma}[\cite{LiMa14-divisorcap}]\label{lem: non-degenerate intersection form}  
	$\omega|_{Y_D}$ is exact if and only if there is a solution for $z$ to the equation $Q_Dz=a$, where $a=([\omega]\cdot [C_1],\dots,[\omega]\cdot [C_r])$ is the area vector. 
	In particular, this holds if  $Q_D$ is non-degenerate.
\end{lemma}
One can also check by simple linear algebra that the above condition is preserved under toric equivalence. If $ \omega|_{Y_D} $ is exact, then there is the following dichotomy depending on whether $ D $ is negative definite.

\begin{theorem}  [\cite{LiMa14-divisorcap}] \label{thm:concave}
	Let $D \subset (W,\omega_0)$ be a symplectic divisor.
	\begin{enumerate}[\indent $ (1) $]
		\item If $ Q_D $ is negative definite, then $ D $ has a convex plumbing.
		\item If  $Q_D$ is not negative definite and
		$\omega_0$ restricted to the boundary  of $D$ is exact, 
		then $\omega_0$ can be locally deformed through a family of symplectic forms $\omega_t$ on $W$ keeping $D$  symplectic and such that $(D,\omega_1)$ is a concave divisor.
	\end{enumerate}
	
\end{theorem}

Although the statement of Theorem \ref{thm:divisorcap} concerns the ambient symplectic manifold $ (W,\omega ) $, it actually does not rely on it. Suppose $ D $ is only a topological divisor with intersection matrix $ Q_D $ such that there exists $ z,a $ satisfying the positive (resp. negative) GS criterion $ Q_D z=a $. Then Theorem \ref{thm:divisorcap} actually constructs a compact concave (resp. convex) symplectic manifold $ (N_D,\omega(z)) $ such that $ D $ is $ \omega(z)- $orthogonal symplectic divisor in $ N_D $ and $ a $ is the $ \omega(z)- $area vector of $ D $. 

Now that $(N_D,\omega(z))$ is a convex or concave neighborhood of $D$, the Liouville vector field induces a contact structure on the boundary which we call $\xi_D$. 
\begin{remark}\label{rmk:positive contact str}
	In this paper we always require a contact structure to be positive, i.e. $\alpha\wedge d\alpha >0$ for any contact form $\alpha$. Note that $Y_D$ is oriented as the boundary of $N_D$. In the case of convex neighborhood, $(Y_D,\xi_D)$ is a positive contact manifold. But if $N_D$ is a concave neighborhood, we have $\alpha\wedge d\alpha <0$ in this orientation as the Liouville vector field points inward. So $(-Y_D,\xi_D)$ is the correct positive contact manifold when $N_D$ is concave.
\end{remark}
The following uniqueness result implies that the symplectic structure $ \omega(z) $ may vary with $ z $ but the induced contact structure on the boundary only depends on $ D $.

\begin{prop}[\cite{LiMa14-divisorcap}, cf. \cite{McL16-discrepancy}]\label{prop:unique-contact}
	Suppose $ D $ is an $ \omega- $orthogonal symplectic divisor which satisfies the positive/negative GS criterion. Then the contact structures induced on the boundary are contactomorphic, independent of choices made in the construction and independent of the symplectic structure $ \omega  $, as long as $ (D,\omega ) $ satisfies positive/negative GS criterion.
	
	Moreover, if $ D $ arises from resolving an isolated normal surface singularity, then the contact structure induced by the negative GS criterion is contactomorphic to the contact structure induced by the complex structure.
\end{prop}

This motivates us to consider the notion of convexity for topological divisors. A topological divisor $ D $ is called {\bf concave} (resp. {\bf convex}) if there exists $ z\in (\RR_+)^r $ (resp. $ z\in (\RR_{\le 0})^r $) such that $ a=Q_Dz\in (\RR_+)^r $. Then there is a contact manifold $ (-Y_D,\xi_D) $ (resp. $ (Y_D,\xi_D) $) and its symplectic cap (resp. filling) $ (N_D,\omega(z)) $ containing $ D $ as a symplectic divisor. One can check by simple linear algebra that being concave (resp. convex) is preserved by toric blow-up (see for example Lemma 3.8 of \cite{LiMa14-divisorcap}).

When $ D $ is negative definite, $ (Y_D,\xi_D ) $ is contactomorphic to the contact boundary of some isolated surface singularity (\cite{Grauert1962}) and is called a Milnor fillable contact structure. A closed 3-manifold $ Y $ is called Milnor fillable if it carries a Milnor fillable contact structure. For every Milnor fillable $ Y $, there is a unique Milnor fillable contact structure (\cite{Caubel-Nemethi-Pampu2006}), i.e. the contact structure $ \xi_D  $ only depends on the oriented homeomorphism type of $ Y_D $ instead of $ D $ when $ D $ is negative definite. 

In light of this uniqueness result, it is natural to ask if similar results hold when $ D $ is concave. The first and the third authors formulated the following question in \cite{LiMi-ICCM}.


\begin{question}[\cite{LiMi-ICCM}]
	Suppose $ D_1 $ and $ D_2 $ are concave divisors with $ -Y_{D_1}\cong -Y_{D_2} $. 
	Suppose either $ b^+(Q_{D_1})=b^+(Q_{D_2}) $ or $\xi_{D_1},\xi_{D_2}$ both symplectically fillable, then is $ (-Y_{D_1},\xi_{D_1}) $ contactomorphic to $ (-Y_{D_2},\xi_{D_2}) $?
\end{question}
Since the torus bundles in Theorem \ref{thm:list} are all distinct, we get a positive answer to the question if we restrict ourselves to fillable concave circular spherical divisors.


As a first step towards this quesition, we have the following invariance result.
\begin{prop}[\cite{LiMi-ICCM}]\label{prop:contact toric eq}
	The contact structure induced by GS construction is invariant under toric equivalence.
\end{prop}

In light of Proposition \ref{prop:contact toric eq}, we see that toric equivalence is a natural equivalence on divisors. For the study of contact structures and symplectic fillings, it suffices to consider toric minimal divisors. In particular it is used in the proof of Theorem \ref{thm:embeddable=fillable} (see Theorem \ref{thm:fillability}).


%

\subsection{Trichotomy}\label{section:trichotomy}
In this subsection, we characterize the convexity of circular spherical divisors. We start with the following lemma about more general topological divisors. By a subdivisor $ D' $ of a topological divisor $ D $, we mean the union of a subset of the components so that $ D' $ is still connected.

\begin{lemma}\label{lemma:subdivisor}
	A topological divisor $ D $ is concave if it contains a concave subdivisor $ D' $.
\end{lemma}
\begin{proof}
	From $ D' $ we can always get back to $ D  $ by successively adding components to $ D' $, so without loss of generality we could assume $ D=\cup_{i=1}^{r} C_i $ can be obtained from $ D'=\cup_{i=1}^{r-1} C_i $ by adding exactly one component $ C_r $.
	
	Note that the intersection matrix $ Q_{D'}=(q_{i,j}) $ of $ D' $ is a $ (r-1)\times (r-1) $ symmetric matrix with nonnegative off-diagonal entries. Then we could write \[
	Q_D=\left( \begin{array}{@{}c|c@{}}
	& q_{1,r}\\
	\raisebox{-5pt}{{\huge\mbox{{$Q_{D'}$}}}} & \vdots\\
	& q_{r-1,r}\\ \hline 
	q_{r,1}\dots q_{r,r-1} & q_{r,r}
	\end{array}\right),
	\]
	where $ q_{i,r}=q_{r,i}\ge 0 $ for $ i=1,\dots,r-1 $ and at least one of them is strictly positive. Assume there exists $ z'\in (\RR_+)^{r-1} $ such that $ a'=Q_{D'}z'\in (\RR_+)^{r-1} $ and let $ z $ be a vector in $ (\RR_+)^r $ such that $ z_i=z_i' $ for $ i=1,\dots,r-1 $. Let $ a=Qz $. Then we have $ a_i=a_i'+q_{i,r}z_r\ge a_i'>0 $ for $ i\le r-1 $ and $ a_r=\sum_{i\le r-1} q_{r,i}z_i' + q_{r,r}z_r $. Since there is at least one strictly positive $ q_{r,i} $, we have $ \sum_{i\le r-1} q_{r,i}z_i'>0 $. If $ q_{r,r}\ge 0 $, then for any $ z_r>0 $ we have $ a_r>0 $. If $ q_{r,r}<0 $, then we can choose $ z_r\in (0,-\dfrac{\sum_{i\le r-1} q_{r,i}z_i'}{q_{r,r}}) $ and get $ a_r>0 $.
\end{proof}
\begin{cor}\label{cor:concave criterion}
	A topological divisor with at least two components and at least one component with nonnegative self-intersection must be concave.
\end{cor}
\begin{proof}
	Such a divisor would have a subdivisor with intersection matrix looking like \[
	Q=\begin{pmatrix}
	a & n\\ n & b
	\end{pmatrix},
	\]
	where $ a\ge 0 $ and $ n\ge 1 $. It suffices to find $ z\in (\RR_+)^2 $ such that $ Qz\in (\RR_+)^2 $.
	If $ b\ge 0 $, we can take $ z=(1,1)^T $. If $ b<0 $, we can take $ z=([-\dfrac{b}{n}+1],1)^T $. So this subdivisor is concave and thus the original divisor is concave by Lemma \ref{lemma:subdivisor}.
\end{proof}

Generalizing the symplectic Kodaira dimension of a symplectic 4-manifold, the following contact Kodaira dimension for contact 3-manifolds was proposed in \cite{LiMa16-kodaira}, based on the type of symplectic cap it admits.
\begin{definition}[\cite{LiMa16-kodaira}, \cite{LiMaYa14-CYcap}]\label{def:contact kod}
	Let $(P, \omega)$ be a  concave symplectic 4-manifold with contact boundary $(Y, \xi)$. $(P, \omega)$  is called a Calabi-Yau  cap of $(Y, \xi)$ if $c_1(P,\omega)$ is a torsion class,  
	and it is called a uniruled cap of $(Y, \xi)$ if there is a contact primitive $\beta$ on the boundary such that $c_1(P,\omega)\cdot  [(\omega, \beta)]>0$.

	The contact Kodaira  dimension   of a contact 3-manifold $(Y, \xi)$ is  defined in terms of uniruled caps and Calabi-Yau caps. 
	Precisely, $Kod(Y, \xi)=-\infty$ if $(Y, \xi)$  has a uniruled cap, $Kod(Y, \xi)=0$ if it has a Calabi-Yau cap but no uniruled caps, $Kod(Y, \xi)=1$ if it has no Calabi-Yau caps or uniruled caps. 
	
\end{definition}

We have the following characterization of convexity for circular spherical divisors and contact Kodaira dimension of their boundaries.
\begin{prop}\label{prop: convex-concave}
	Let $ D $ be a circular spherical divisor and $ Q_D $ its intersection matrix.
	\begin{enumerate}[\indent $ (1) $]
		\item If $Q_D$ is negative definite, then $D$ is convex. In addition, if $D$ is anti-canonical in $ (X,\omega) $, 
		then $(Y_D, \xi_D)$ has  $Kod(Y_D,\xi_D)\leq 0$. 
		\item If $b^+(Q_D)> 0$, then $ D $ is concave. In addition, if $ D $ is symplectically embeddable, then 
		$(-Y_D, \xi_D)$ has  $Kod(-Y_D,\xi_D)=-\infty$.
		\item If $b^+(Q_D)=0$ and $Q_D$ is not negative definite, then $ D $ is neither concave or convex.
	\end{enumerate}
\end{prop}
\begin{proof}
	Since being concave or convex is preserved under toric blow-up, we could assume $ D $ is either toric minimal or of the form $ (-1,p) $ and make use of the classification in Lemma \ref{lem: not negative semi-definite => at least one non-negative}. 
	
	{\bf Case (1)}: Suppose $ Q_D $ is negative definite. Then for any $ a\in (\RR_+)^r $, there is a unique solution $ z\in (\RR_-)^r $ such that $ Q_Dz=a $. So $ D $ is convex. Since $D$ is anti-canonical, let $ (X,\omega,D) $ be a symplectic Looijenga pair. Then by Theorem \ref{thm:concave} there is a convex neighborhood $ N_D $ of $ D $ in $ (X,\omega) $ and $P=X-N_D$ is a symplectic cap of $Y_D$ with vanishing $c_1$, i.e. a Calabi-Yau cap. It follows that $Kod(Y_D, \xi_D)\leq 0$. 
	
	{\bf Case (2)}: Suppose $ b^+(Q_D)>0 $. If $ D $ is toric minimal, there exists a component $ C_j $ with non-negative self-intersection. By Corollary \ref{cor:concave criterion}, $ D $ must be concave. If $ D $ is of the form $ (-1,p) $, we must have $ p\ge -3 $. Then for $ z=(7,4)^T $, we have $ Q_Dz=(1,14+4p)^T\in (\RR_+)^2 $, i.e. $ D $ is concave. 
	If $D$ is symplecticaly embeddable, then $D$ is anti-canonical by Theorem \ref{thm:embeddable=rigid}.
	Let $ (X,\omega,D) $ be a symplectic Looijenga pair and assume $ \omega|_{Y_D} $ is exact, then by Theorem \ref{thm:concave} there is a concave neighborhood $ N_D $ of $ D $ in $ (X,\omega) $ with boundary $ (-Y_D,\xi_D) $, up to a local symplectic deformation. For any contact primitive $\alpha$ of $\omega|_{Y_D}$,  we have 
	$c_1(N_D)\cdot [(\omega,\alpha)]=   c_1(X)|_{N_D} \cdot [(\omega,\alpha)] = D\cdot  [(\omega, \alpha)]=D\cdot [\omega]>0$, i.e. $ N_D $ is a uniruled cap for $ (-Y_D,\xi_D) $.
	
	Now we show that $ \omega|_{Y_D} $ is exact. By Lemma \ref{lem: non-degenerate intersection form}, it suffices to show that for any $ a\in (\RR_+)^r $, there is a solution $ z\in \RR^r $ to the equation $ Qz=a $. And it suffices to show it for all $ D $ listed in Theorem \ref{thm:list}. By Lemma \ref{lem: property of continuous fraction}, we only need to look at the positive parabolic case, because $ Q_D $ is nondegenerate in all other cases. Let $ a=(a_1,a_2,a_3)=([\omega]\cdot [C_1],[\omega]\cdot [C_2],[\omega]\cdot [C_3]) $ for the divisor $ (1,1,p) $. First if $ p=1 $, then by (4) of Lemma \ref{lemma:topological cyclic}, we have $ [C_1]=[C_2]=[C_3] $ and thus $ a_1=a_2=a_3 $. Then $ z_1=z_2=z_3=\dfrac{a_1}{3} $ is a solution to $ Q_Dz=a $. If $ p\neq 1 $, then again by (4) and (5) of Lemma \ref{lemma:topological cyclic}, we have $ [C_1]=[C_2] $ and thus $ a_1=a_2 $. Solving the equation $ Q_Dz=a $	is thus equivalent to solving \[
	\begin{pmatrix}
	1&1\\1&p
	\end{pmatrix}\begin{pmatrix}
	z_1\\z_3
	\end{pmatrix}=\begin{pmatrix}
	a_1\\a_3
	\end{pmatrix},
	\]
	which is solvable as $ p\neq 1 $. 
	
	{\bf Case (3)}: Suppose $ b^+(Q_D)=0 $ and $ Q_D $ is not negative definite. If $ D $ is toric minimal, it is a cycle of self-intersection $-2$ spheres. Note that the $ i $th-entry of $ Q_{D}z $ is $ (z_{i+1}-z_i)-(z_i-z_{i-1}) $ for all $ i $, where the index is taken to be modulo $ r $. So it is easy to see that for any $z \in \mathbb{R}^{r}$, $Q_{D}z$ cannot have all entries being positive. If $ D $ is of the form $ (-1,p) $, we must have $ p=-4 $. Similarly, $ Q_Dz=(-(z_1-2z_2),2(z_1-2z_2))^T $ can not be in $ (\RR_+)^2 $ for any $ z $. So $ D $ cannot be convex or concave.
\end{proof}

\subsection{Symplectic fillings for concave $ D $}\label{section:fillability}


In this section we prove Theorem \ref{thm:embeddable=fillable}, which is splitted into Theorem \ref{thm:fillability} and \ref{thm: as a support of ample line bundle}. 
We start by collecting some homological information related to a cycle of spheres in a closed 4-manifold $ X $ with $b_1(X)=1$, which will also be useful in Section \ref{section:convex Stein}.
\begin{lemma}\label{lemma:U+V}
	Let $X=U\cup_Y V$ be a closed 4-manifold obtained by gluing $U$ and $V$ along their common boundary $Y$. Suppose $b_3(U)=0$ and $b_1(X)=0$. Then we have $b_3(V)=0$ and $b^0(V)+b_1(V)=b^0(U)+b_1(U)=b_1(Y)$.
\end{lemma}
\begin{proof}
	Since $H_4(X)\to H_3(Y)$ is an isomorphism, we have $0\to H_3(V;\QQ)\oplus H_3(U;\QQ)\to H_3(X;\QQ)=0$ and thus $b_3(V)=0$. Note that $b^0(V)$ is the kernel dimension of the map $ H_2(V;\QQ)\to H_2(V,Y;\QQ) $. Then we have the long exact sequence \[
		0\to \QQ^{b^0(V)}\to H_2(V;\QQ)\to H_2(V,Y;\QQ)\to H_1(Y;\QQ) \to H_1(V;\QQ)\to 0,
	\]
	which gives $ b^0(V)+b_1(V)=b_1(Y) $ by Lefschetz duality.
\end{proof}

Recall that the charge $q(D)=12-D^2-r(D)$ is invariant under toric equivalence. 
\begin{lemma} [cf. Theorem 2.5 and Theorem 3.1 in \cite{GoLi14}]  \label{lem: homology of neighborhood} 
	Let  $D$ be a cycle of spheres in a closed 4-manifold $X$ with $b_1(X)=0$ and $V=X-N_D$, where $N_D$ is a plumbing of $D$. Then we have 
	\[H_2(N_D)=\Z^{r(D)}=H^2(N_D), H_1(N_D)=H^1(N_D)=\Z,  H_3(N_D)=H^3(N_D)=0.\]
	If $Q_D$ is nonsingular, then we have \[b_1(Y_D)=1,b_1(V)=0,b^0(V)=1,b_2(V)=b_2(X)+1-r(D).\]
	If $(X,D,\omega)$ is a symplectic Looijenga pair, then $e(V)=q(D)$ and $\sigma(V)=4-q(D)-2b^+(D)-b^0(D)$.
\end{lemma}

\begin{proof}
	The homology and cohomology of $N_D$ are straightforward to compute since $N_D$ deformation retracts to $D$. 
	
	As in Theorem 2.5 of \cite{GoLi14}, we have $H_1(Y_D)=\ZZ\oplus Coker(Q_D)$.	
	Now suppose $Q_D$ is nonsingular. Then $b_1(Y)=1$.
	By Mayer-Vietoris sequence \[
		H_1(Y;\QQ)\to H_1(N_D;\QQ)\oplus H_1(V;\QQ)\to 0, \]
	we get that $b_1(Y)\ge 1+b_1(V)$. So $b_1(V)=0$ and then $b^0(V)=1$, since $ b^0(V)+b_1(V)=b_1(Y) $ by Lemma \ref{lemma:U+V}. 
	Since $e(V)=e(X)-e(N_D)=2+b_2(X)-r(D)$ and by definition $e(V)=1+b_2(V)$, we have $b_2(V)=b_2(X)+1-r(D)$.

	If $(X,D,\omega)$ is a symplectic Looijenga pair, then $e(X)=12-D^2$ and $\sigma(X)=D^2-8$. Note $e(N_D)=r(D)$ and $\sigma(N_D)=b^+(D)-b^-(D)$. By additivity, we get $e(V)=12-D^2-r(D)=q(D)$ and $\sigma(V)=D^2-8-\sigma(N_D)=D^2-12+r(D)-2b^+(D)-b^0(D)+4=4-q(D)-2b^+(D)-b^0(D)$.
\end{proof}

\begin{theorem}\label{thm:fillability}
	Let $ D $ be a concave circular spherical divisor. Then $ (-Y_D,\xi_D) $ is symplectic fillable if and only if $ D $ is toric equivalent to one of Theorem \ref{thm:list}. Moreover, there are finitely many minimal symplectic fillings. All minimal symplectic filling $(W,\omega_W)$ of such $ (-Y_D,\xi_D) $ have $ c_1(W,\omega_W)=0$. They all have $b^+(W)=b_3(W)=0$, $b^0(W)=1$, $b_1(W)=b_1(Y_D)-1$, $b^-(W)=q(D)-2+b^0(D)$ and hence a unique rational homology type. In particular, \begin{itemize}
		\item for $D$ in case (1)(3)(4), we have $b^-(W)=q(D)-2$, $b_2(W)=q(D)-1$ and $b_1(W)=1$;
		\item for $D$ in case (2) with $p=1$, we have $b^-(W)=0$, $b_2(W)=1$ and $b_1(W)=2$;
		\item for $D$ in case (2) with $p<1$, we have $b^-(W)=q(D)-1$, $b_2(W)=q(D)$ and $b_1(W)=1$.
	\end{itemize}
\end{theorem}
\begin{proof}
	If $ D $ is toric equivalent to one in Theorem \ref{thm:list}, then it admits an anticanonical symplectic embedding into a symplectic rational surface $ (X,\omega ) $. By Proposition \ref{prop: convex-concave}, up to local symplectic deformation, there is a concave symplectic neighborhood $ N_D $ of $ D $ with contact boundary $ (-Y_D,\xi_D) $, then $ X-Int(N_D) $ is a symplectic filling of $ (-Y_D,\xi_D) $.
	
	Now suppose $ (-Y_D,\xi_D ) $ is symplectic fillable and let $ (W,\omega_W ) $ be any minimal symplectic filling. Let $ z,a $ be a pair of vectors satisfy the positive GS criterion $ Q_Dz=a $. Then we have a divisor cap $ (N_D,\omega(z)) $ of $ (-Y_D,\xi_D) $. Glue $ (W,\omega_W ) $ with $ (N_D,\omega(z)) $ to get a closed manifold $ (X,\omega ) $. By Lemma \ref{lemma:b^+ leq 1} and \ref{lemma:rational-embed}, we have that $ b^+(D)=1 $ and $ X $ is a rational surface and $ D $ is toric equivalent to one listed in Theorem \ref{thm:list}. 
	
	Since $ (W,\omega_W) $ is a minimal symplectic filling, there is no smoothly embedded sphere of self-intersection $ -1 $ in $ Int(W) $ by Proposition 3.1 in \cite{LiMa16-kodaira}. So $ X-D $ is minimal since it is diffeomorphic to $ X-N_D=Int(W) $.
	Because $ D $ is a rationally embeddable circular divisor with $ b^+(D)= 1 $, $ D $ is rigid and thus $ (X,D,\omega ) $ is a symplectic Looijenga pair. So we must have $ c_1=b^+=b_3=0 $ for $ (W,\omega_W) $ by Lemma \ref{lemma:U+V}. The finiteness of symplectic deformation types follows from Corollary \ref{cor: finite deformation}.

	Now we compute other Betti numbers of $ W $. By Lemma \ref{lem: homology of neighborhood}, we have $e(W)=q(D)$ and $\sigma(W)=2-q(D)-b^0(D)$.
	Since $b^+(W)=0$, we have that $b^-(W)=-\sigma(W)=q(D)+b^0(D)-2$. By Lemma \ref{lemma:toric eq topological}, $b^0(D)$ is invariant under toric equivalence. Then Proposition \ref{prop:unique-contact} implies that both $q(D)$ and $b^0(D)$ depend only on $(-Y_D,\xi_D)$. So all minimal symplectic fillings have the same $b^\pm(W), b^0(W)$ and a unique rational homology type.

	When $Q_D$ is nonsingular, we have by Lemma \ref{lem: homology of neighborhood} that $b_1(W)=0$, $b_2(W)=b_2(X)+1-r(D)=10-D^2-r(D)+1=q(D)-1$ and also $b^-(W)=q(D)-2$. 
	
	When $Q_D$ is singular, we have $D$ is toric equivalent to $(1,1,p)$ with $p\le 1$. 
	By Lemma \ref{lemma:U+V}, we still have \[
		1-b_1(W)+b^0(W)+b^-(W)=e(W)=q(D) \text{ and }	b^0(W)-b_1(Y)+b_1(W)=0.\]
	These two equations imply that \[
	b^0(W)=\dfrac{1}{2}(b_1(Y)+q(D)-b^-(W)-1)\text{ and } b_1(W)=\dfrac{1}{2}(b_1(Y)-q(D)+b^-(W)+1).	
	\]
	Note again that toric equivalence preserves $b^+(D),b^0(D)$ by Lemma \ref{lemma:toric eq topological}. \begin{itemize}
		\item When $p=1$ we have $b_1(Y_D)=3$, $b^0(D)=2$, and thus $b^-(W)=0$. Then $b^0(W)=1$ and $b_1(W)=2$.
		\item When $p<1$ we have $b_1(Y_D)=2$, $b^0(D)=1$, and thus $b^-(W)=q(D)-1$. Then $b^0(W)=b_1(W)=1$.
	\end{itemize}  
\end{proof}

Next we identify minimal symplectic fillings with Stein fillings up to symplectic deformation equivalence. This finishes the proof of Theorem \ref{thm:embeddable=fillable}.
\begin{theorem}\label{thm: as a support of ample line bundle}
	For  a symplectic Looijenga pair $(X, D, \omega)$ with 
	$b^+(Q_D)= 1$,     there exists a K\"ahler Looijenga pair $(\overline{X},\overline{D},\overline{\omega})$ in its  symplectic deformation  class
	such that $\overline{D}$ is the support of an ample line bundle. 
	Then every minimal symplectic fillings of $ (-Y_D,\xi_D ) $ is symplectic deformation equivalent to a Stein filling.
\end{theorem}

\begin{proof}
	
	Let $(X,D,\omega)$ be a symplectic Looijenga pair. By Lemma \ref{thm: symplectic deformation class=homology classes} there is an holomorphic Looijenga pair $(\overline{X}',\overline{D}',\overline{\omega}')$ symplectic deformation equivalent to $ (X,D,\omega ) $.
	By Proposition  \ref{prop: convex-concave}, $\overline{\omega}'$ is exact on $\partial P(\overline{D})$ and there is another symplectic form $ \overline{\omega} $ deformation equivalent to $ \overline{\omega }' $ such that $Q_D z = a$, where $ z=(z_1,\dots,z_r)\in (\RR_+)^r $ and $a=([\overline{\omega}]\cdot [\overline{C_1}],\dots,[\overline{\omega}]\cdot [\overline{C_r}])$. It means that $\tilde{D}=\sum\limits_{i=1}^r z_i[\overline{C_i}]$ pairs positively with all $[\overline{C_i}]$ and in particular $ \tilde{D}^2>0 $.

	Moreover, by Proposition 4.1 of \cite{GrHaKe12}, we can choose a complex structure 
	such that $(\overline{X},\overline{D})$ is a generic pair, which means there is no smooth rational curve of self-intersection $ -2 $ disjoint from $ D $ (Definition 1.4 of \cite{GrHaKe12}).
	Any algebraic curve $ C $ disjoint from $ \overline{D} $ is in particular orthogonal to $ \tilde{D} $. By Hodge index theorem, we must have $ [C]\le 0 $. Then $ C $ is a self-intersection $-2$ rational curve or a self-intersection $0$ elliptic curve by adjunction formula.
	Note that $ C $ cannot be a self-intersection $-2$ rational curve by the genericity of the pair $(\overline{X},\overline{D})$.
	We claim that there is no elliptic curve in the complement of $ \overline{D} $ of self-intersection $0$. Suppose there exists such elliptic curve $ T $. Since $ b^+(Q_D)=1 $, we can assume there is a self-intersection $0$ component $ \overline{C} $ of $\overline{D}$ (possibly after toric blow-down and non-toric blow-up). Using light cone lemma, we get that $ [T]=\lambda [\overline{C}] $ for some $ \lambda > 0 $. However, $ \overline{C} $ intersects other components of $ \overline{D} $ non-trivially, and so would $ T $, which is contradiction.
	Any algebraic curve that intersects $\overline{D}$ but not contained in $\overline{D}$ has positive pairing with $\sum_{i=1}^r z_i[\overline{C_i}]$.
	Also, by the choice of $\sum_{i=1}^r z_i[\overline{C_i}]$, it pairs positively with any irreducible curve in $\overline{D}$.
	Therefore, by Nakai-Moishezon criterion, $\sum_{i=1}^r z_i[\overline{C_i}]$ is an ample divisor. Then some multiple of it would be very ample with support $\overline{D}$.
	So $ \overline{X}-\overline{D} $ is an affine surface and provides a Stein filling of $ (-Y_D,\xi_D) $.
	
	As shown in the proof of Theorem \ref{thm:fillability}, every minimal symplectic filling gives rise to a symplectic Looijenga pair $(X,D,\omega)$, which is symplectic deformation equivalent to a K\"ahler pair $(\overline{X},\overline{D},\overline{\omega})$ as above. So the minimal symplectic filling is symplectic deformation equivalent to a Stein filling.
\end{proof}


\begin{remark}\label{rmk:ohta-ono}
	A similar result for elliptic log Calabi-Yau pairs was obtained by Ohta and Ono in \cite{OhOn03-simple-elliptic}. Their results were stated for links of simple elliptic singularities, but actually concerns symplectic torus of positive self-intersection. We summarize their results as follows.\\
	Let $ D $ be a torus with $ 0<[D]^2 \le 9 $ and $ [D]^2\neq 8 $, then the minimal symplectic filling of $ (-Y_D,\xi_D ) $ is unique up to diffeomorphism. In the case $ [D]^2=8 $, there are two diffeomorphism types of minimal symplectic fillings. All these minimal fillings have $ c_1=0 $ and $ b^+=0 $. When $ [D]^2\ge 10 $, there is a unique minimal symplectic filling up to diffeomorphism, but we don't have $ c_1=0 $ in this case.
\end{remark}
%
%
%
%




\subsection{Geography of Stein fillings for convex anti-canonical $ D $}\label{section:convex Stein}

When $ Q_D $ is negative definite, the circular spherical divisor $D$ corresponds to resolutions of cusp singularities. Then $ (Y_D,\xi_D) $ is a Milnor fillable contact structure and is thus Stein fillable (\cite{Grauert1962}, see also \cite{Caubel-Nemethi-Pampu2006}).


The elliptic log Calabi-Yau pairs arise from resolutions of simple elliptic singularities. By Theorem 2 in \cite{OhOn03-simple-elliptic}, any simple elliptic singularity has either one or two minimal symplectic fillings up to diffeomorphism, arising either from a smoothing or the minimal resolution. 

It is then natural to ask what kind of finiteness holds for symplectic fillings of $(Y_D,\xi_D)$ when $D$ is convex anti-canonical circular spherical divisor. Since $(Y_D,\xi_D)$ has non-positive contact Kodaira dimension (Proposition \ref{prop: convex-concave}), the Betti numbers of exact fillings of $(Y_D, \xi_D)$ are bounded by Theorems  1.3 and 1.8 in \cite{LiMaYa14-CYcap}. Here we provide explicit Betti number bounds for their Stein fillings.

\begin{prop}\label{prop:geography}
	Let $D$ be a convex anti-canonical circular spherical divisor, i.e. $ Q_D $ negative definite. Let $(U,\omega_U)$ be a Stein filling of $(Y_D,\xi_D)$, then 
	\begin{enumerate}[\indent $ (1) $]
		\item $U$ is negative definite with $b_1(U)=1$,
		\item or $(b_2^+(U),b_2^0(U),b_1(U))= (1,1,0) $, $c_1(U)=0$, $b^-(U)=21-q(D)\leq 18$, and $3\le q(D)\le 21$,
		\item or $(b_2^+(U),b_2^0(U),b_1(U))= (2,0,1)$, $c_1(U)=0$, $b^-(U)=22-q(D)\leq 19$, and $3\le q(D)\le 22$.
	\end{enumerate}
	In particular, when $q(D)\ge 23$, $U$ must be negative definite.
\end{prop}
\begin{proof}
	Since $D$ is anti-canonical, there exists $(X,\omega)$ such that $(X,D,\omega)$ is a symplectic Looijenga pair. Let $V=X - N_D$ be the symplectic cap obtained as the complement of $ N_D $ and $ X_U=U\cup V $ the closed symplectic manifold obtained by capping off the filling $ U $ with $ V $. 
	
	Since $U$ is Stein and in particular a $ 2 $-handlebody, we have $1=b_1(Y_D) \ge b_1(U)$. 
	The Mayer-Vietoris sequence 
	$H_1(U) \oplus H_1(V) \to H_1(X_U)\to 0$ implies that $b_1(X_U)\le b_1(U)+b_1(V)\le 1+b_1(V)$. By Lemma \ref {lem: homology of neighborhood}, $b_1(V)=0$, so $b_1(X_U) \leq 1$. 
	Since $V$ is a Calabi-Yau cap as $c_1(V)=0$, $ X_U $ must be a symplectic non-minimal rational or minimal Calabi-Yau surface by Lemma 2.8 of \cite{LiMaYa14-CYcap}. It then follows from $b_1(X_U)\leq 1$ and  Theorem 1.1 of \cite{Li06-betti} that 
	\begin{enumerate}[\indent $ (1) $]
		\item either $b^+(X_U)=1, b_1(X_U)=0$ and $X_U$ is non-minimal rational or a minimal integral homology Enrique surface,
		\item or $b^+(X_U)=3, b_1(X_U)=0$ and $X_U$ is a minimal integral homology $K3$ surface.
	\end{enumerate}


	So we always have $b_1(X_U)=0$ and by Lemma \ref{lemma:U+V}, $b^0(U)+b_1(U)=b_1(Y_D)=1$. By Lemma \ref{lem: homology of neighborhood}, we also have $ e(V)=q(D)$ and $\sigma(V)=\sigma(V)=4-q(D)-2b^+(D)-b^0(D)=-q(D)+4 $ since $D$ negative definite.
	By additivity, we get \begin{align*}
		b^+(U)-b^-(U)= &\sigma(U)=\sigma(X_U)+q(D)-4,\\
		b^+(U)+b^-(U)= & e(U)-2b^0(U)=b_2(X_U)+2-q(D)-2b^0(U).
	\end{align*}
	Then we have \begin{align*}
		b^+(U)= & b^+(X_U)-b^0(U)-1,\\
		b^-(U)= & 3-q(D)-b^0(U)+b^-(X_U).
	\end{align*}
	Note that if $b^+(X_U)=3$, then $c_1(X_U)=0$ and thus $c_1(U)=0$. By Lemma 4.3 of \cite{FM83}, we have that $q(D)\ge 3$.
	\begin{itemize}
		\item If $b^0(U)=1$, then $b^+(X_U)\ge 2$, which implies that $X_U$ must be an integer homology K3 with $b^+(X_U)=3$ and $b^-(X_U)=19$. Then we have $c_1(U)=0$, $b_1(U)=0$, $b^+(U)=1$, $b^-(U)=21-q(D)\le 18$ and $3\le q(D)\le 21$.
		\item If $b^0(U)=0$ and $b^+(X_U)=3$, we have $c_1(U)=0$, $b_1(U)=1$, $b^+(U)=2$, $b^-(U)=22-q(D)\le 19$ and $3\le q(D)\le 22$.
		\item If $b^0(U)=0$ and $b^+(X_U)=1$, we have $U$ is negative definite, $b_1(U)=1$, and $e(U)=b^-(U)=3+b^-(X_U)-q(D)$.
	\end{itemize}
	
\end{proof}

Finally, we discuss the potential implication of Proposition \ref{prop:geography} for Stein fillings of   cusp singularities. 
By Looijenga conjecture, a cusp singularity is smoothable if and only if it has a rational dual.
There is an analytic space $ X $ containing a cusp singularity and its dual. It is shown in Proposition 2.8 of \cite{Lo81} that if we smooth the cusp and resolve its dual cusp, we then get a rational surface. Since the resolution part is negative definite, we conclude that a smoothing of a cusp singularity gives a Stein filling of $ (Y_D,\xi_D) $ with $ b^+=1 $. 
In light of this, we make the following symplectic/contact analogue of the Looijenga conjecture.
\begin{conjecture}
	If a  cusp  singularity  does not have  a rational dual, then  it admits only negative definite Stein fillings.
\end{conjecture}
By Theorem 4.5 of \cite{FM83}, if $q(D)\ge 22$ then $D$ doesn't have a rational dual. Then Proposition \ref{prop:geography} proves the conjecture for rational cups singularities with $q(D)\ge 23$. In particular, given any negative-define anti-canonical divisor,  if we perform $20$ non-toric blow-ups (to increase the charge while remaining negative-definite and anti-canonical), the resulting contact torus bundle has only negative definite Stein fillings. 


\bibliographystyle{amsplain}
\bibliography{logCY-sequence}{}

\providecommand{\bysame}{\leavevmode\hbox to3em{\hrulefill}\thinspace}
\providecommand{\MR}{\relax\ifhmode\unskip\space\fi MR }
\providecommand{\MRhref}[2]{%
  \href{http://www.ams.org/mathscinet-getitem?mr=#1}{#2}
}
\providecommand{\href}[2]{#2}
\begin{thebibliography}{10}

\bibitem{Barth15-compact}
Wolf~P. Barth, Klaus Hulek, Chris A.~M. Peters, and Antonius Van~de Ven,
  \emph{Compact complex surfaces}, second ed., Ergebnisse der Mathematik und
  ihrer Grenzgebiete. 3. Folge. A Series of Modern Surveys in Mathematics
  [Results in Mathematics and Related Areas. 3rd Series. A Series of Modern
  Surveys in Mathematics], vol.~4, Springer-Verlag, Berlin, 2004. \MR{2030225}

\bibitem{BhOz14}
Mohan Bhupal and Burak Ozbagci, \emph{Canonical contact structures on some
  singularity links}, Bull. Lond. Math. Soc. \textbf{46} (2014), no.~3,
  576--586. \MR{3210714}

\bibitem{Biran-packing}
P.~Biran, \emph{Symplectic packing in dimension {$4$}}, Geom. Funct. Anal.
  \textbf{7} (1997), no.~3, 420--437. \MR{1466333}

\bibitem{Caubel-Nemethi-Pampu2006}
Cl\'{e}ment Caubel, Andr\'{a}s N\'{e}methi, and Patrick Popescu-Pampu,
  \emph{Milnor open books and {M}ilnor fillable contact 3-manifolds}, Topology
  \textbf{45} (2006), no.~3, 673--689. \MR{2218761}

\bibitem{CD-Enriques}
Fran\c{c}ois~R. Cossec and Igor~V. Dolgachev, \emph{Enriques surfaces. {I}},
  Progress in Mathematics, vol.~76, Birkh\"{a}user Boston, Inc., Boston, MA,
  1989. \MR{986969}

\bibitem{DiGe01-fillability}
Fan Ding and Hansj\"{o}rg Geiges, \emph{Symplectic fillability of tight contact
  structures on torus bundles}, Algebr. Geom. Topol. \textbf{1} (2001),
  153--172. \MR{1823497}

\bibitem{DiLi18-torusbundle}
Fan Ding and Youlin Li, \emph{Strong symplectic fillability of contact torus
  bundles}, Geom. Dedicata \textbf{195} (2018), 403--415. \MR{3820513}

\bibitem{seppi-li-wu-stability}
Josef~G. Dorfmeister, Tian-Jun Li, and Weiwei Wu, \emph{Stability and existence
  of surfaces in symplectic 4-manifolds with {$b^+=1$}}, J. Reine Angew. Math.
  \textbf{742} (2018), 115--155. \MR{3849624}

\bibitem{EbWall1985}
W.~Ebeling and C.~T.~C. Wall, \emph{Kodaira singularities and an extension of
  {A}rnol'd's strange duality}, Compositio Math. \textbf{56} (1985), no.~1,
  3--77. \MR{806842}

\bibitem{Ecap}
Yakov Eliashberg, \emph{A few remarks about symplectic filling}, Geom. Topol.
  \textbf{8} (2004), 277--293. \MR{2023279}

\bibitem{En}
Philip Engel, \emph{Looijenga's conjecture via integral-affine geometry}, J.
  Differential Geom. \textbf{109} (2018), no.~3, 467--495. \MR{3825608}

\bibitem{EHcap}
John~B. Etnyre and Ko~Honda, \emph{On symplectic cobordisms}, Math. Ann.
  \textbf{323} (2002), no.~1, 31--39. \MR{1906906}

\bibitem{Fr}
Robert Friedman, \emph{On the geometry of anticanonical pairs}, Preprint
  arXiv:1502.02560 (2016).

\bibitem{FM83}
Robert Friedman and Rick Miranda, \emph{Smoothing cusp singularities of small
  length}, Math. Ann. \textbf{263} (1983), no.~2, 185--212. \MR{698002}

\bibitem{GaMa13-LF}
David Gay and Thomas~E. Mark, \emph{Convex plumbings and {L}efschetz
  fibrations}, J. Symplectic Geom. \textbf{11} (2013), no.~3, 363--375.
  \MR{3100798}

\bibitem{GoLi14}
Marco Golla and Paolo Lisca, \emph{On {S}tein fillings of contact torus
  bundles}, Bull. Lond. Math. Soc. \textbf{48} (2016), no.~1, 19--37.
  \MR{3455745}

\bibitem{Gom95-fibersum}
Robert~E. Gompf, \emph{A new construction of symplectic manifolds}, Ann. of
  Math. (2) \textbf{142} (1995), no.~3, 527--595. \MR{1356781}

\bibitem{Grauert1962}
Hans Grauert, \emph{\"{U}ber {M}odifikationen und exzeptionelle analytische
  {M}engen}, Math. Ann. \textbf{146} (1962), 331--368. \MR{137127}

\bibitem{GrHaKe11}
Mark Gross, Paul Hacking, and Sean Keel, \emph{Mirror symmetry for log
  {C}alabi-{Y}au surfaces {I}}, Publ. Math. Inst. Hautes \'{E}tudes Sci.
  \textbf{122} (2015), 65--168. \MR{3415066}

\bibitem{GrHaKe12}
\bysame, \emph{Moduli of surfaces with an anti-canonical cycle}, Compos. Math.
  \textbf{151} (2015), no.~2, 265--291. \MR{3314827}

\bibitem{Kod-surface3}
K.~Kodaira, \emph{On compact analytic surfaces. {II}, {III}}, Ann. of Math. (2)
  77 (1963), 563--626; ibid. \textbf{78} (1963), 1--40. \MR{0184257}

\bibitem{Li06-betti}
Tian-Jun Li, \emph{Quaternionic bundles and {B}etti numbers of symplectic
  4-manifolds with {K}odaira dimension zero}, Int. Math. Res. Not. (2006), Art.
  ID 37385, 28. \MR{2264722}

\bibitem{Li06-kod-0}
\bysame, \emph{Symplectic 4-manifolds with {K}odaira dimension zero}, J.
  Differential Geom. \textbf{74} (2006), no.~2, 321--352. \MR{2259057}

\bibitem{LiMa16-deformation}
Tian-Jun Li and Cheuk~Yu Mak, \emph{Symplectic log {C}alabi-{Y}au
  surface---deformation class}, Adv. Theor. Math. Phys. \textbf{20} (2016),
  no.~2, 351--379. \MR{3541847}

\bibitem{LiMa19-survey}
\bysame, \emph{Geometry of symplectic log {C}alabi-{Y}au pairs}, ICCM Not.
  \textbf{6} (2018), no.~2, 42--50. \MR{3961489}

\bibitem{LiMa14-divisorcap}
\bysame, \emph{Symplectic divisorial capping in dimension 4}, J. Symplectic
  Geom. \textbf{17} (2019), no.~6, 1835--1852. \MR{4057728}

\bibitem{LiMa16-kodaira}
\bysame, \emph{The {K}odaira dimension of contact 3-manifolds and geography of
  symplectic fillings}, Int. Math. Res. Not. IMRN (2020), no.~17, 5428--5449.
  \MR{4146343}

\bibitem{LiMaYa14-CYcap}
Tian-Jun Li, Cheuk~Yu Mak, and Kouichi Yasui, \emph{Calabi-{Y}au caps, uniruled
  caps and symplectic fillings}, Proc. Lond. Math. Soc. (3) \textbf{114}
  (2017), no.~1, 159--187. \MR{3653080}

\bibitem{LiMi-logkod}
Tian-Jun Li and Jie Min, \emph{Relative {K}odaira dimension for symplectic
  divisors}, In preparation.

\bibitem{LiMi-ICCM}
Tian-Jun Li and Jie Min, \emph{Local geometry of symplectic divisors with
  applications to contact torus bundles}, Preprint arXiv:2101.05981 (2021).

\bibitem{LiUs06-negative-inflation}
Tian-Jun Li and Michael Usher, \emph{Symplectic forms and surfaces of negative
  square}, J. Symplectic Geom. \textbf{4} (2006), no.~1, 71--91. \MR{2240213}

\bibitem{LiZh11-relative}
Tian-Jun Li and Weiyi Zhang, \emph{Additivity and relative {K}odaira
  dimensions}, Geometry and analysis. {N}o. 2, Adv. Lect. Math. (ALM), vol.~18,
  Int. Press, Somerville, MA, 2011, pp.~103--135. \MR{2882443}

\bibitem{Lis08-lens}
Paolo Lisca, \emph{On symplectic fillings of lens spaces}, Trans. Amer. Math.
  Soc. \textbf{360} (2008), no.~2, 765--799. \MR{2346471}

\bibitem{Lo81}
Eduard Looijenga, \emph{Rational surfaces with an anticanonical cycle}, Ann. of
  Math. (2) \textbf{114} (1981), no.~2, 267--322. \MR{632841}

\bibitem{Mc90-structure}
Dusa McDuff, \emph{The structure of rational and ruled symplectic
  {$4$}-manifolds}, J. Amer. Math. Soc. \textbf{3} (1990), no.~3, 679--712.
  \MR{1049697}

\bibitem{McOp13-nongeneric}
Dusa McDuff and Emmanuel Opshtein, \emph{Nongeneric {$J$}-holomorphic curves
  and singular inflation}, Algebr. Geom. Topol. \textbf{15} (2015), no.~1,
  231--286. \MR{3325737}

\bibitem{McDuffSalamon1996}
Dusa McDuff and Dietmar Salamon, \emph{A survey of symplectic {$4$}-manifolds
  with {$b^{+}=1$}}, Turkish J. Math. \textbf{20} (1996), no.~1, 47--60.
  \MR{1392662}

\bibitem{McL16-discrepancy}
Mark McLean, \emph{Reeb orbits and the minimal discrepancy of an isolated
  singularity}, Invent. Math. \textbf{204} (2016), no.~2, 505--594.
  \MR{3489704}

\bibitem{Miy-kahler}
Yoichi Miyaoka, \emph{K\"{a}hler metrics on elliptic surfaces}, Proc. Japan
  Acad. \textbf{50} (1974), 533--536. \MR{460730}

\bibitem{Mumford1961}
David Mumford, \emph{The topology of normal singularities of an algebraic
  surface and a criterion for simplicity}, Inst. Hautes \'{E}tudes Sci. Publ.
  Math. (1961), no.~9, 5--22. \MR{153682}

\bibitem{Ne81-calculus}
Walter~D. Neumann, \emph{A calculus for plumbing applied to the topology of
  complex surface singularities and degenerating complex curves}, Trans. Amer.
  Math. Soc. \textbf{268} (1981), no.~2, 299--344. \MR{632532}

\bibitem{OhOn03-simple-elliptic}
Hiroshi Ohta and Kaoru Ono, \emph{Symplectic fillings of the link of simple
  elliptic singularities}, J. Reine Angew. Math. \textbf{565} (2003), 183--205.
  \MR{2024651}

\bibitem{OhtOn05-cusp}
\bysame, \emph{Symplectic 4-manifolds containing singular rational curves with
  {$(2,3)$}-cusp}, Singularit\'{e}s {F}ranco-{J}aponaises, S\'{e}min. Congr.,
  vol.~10, Soc. Math. France, Paris, 2005, pp.~233--241. \MR{2145957}

\bibitem{Shioda-modular}
Tetsuji Shioda, \emph{On elliptic modular surfaces}, J. Math. Soc. Japan
  \textbf{24} (1972), 20--59. \MR{429918}

\bibitem{Shioda-elliptic}
\bysame, \emph{Elliptic surfaces and {D}avenport-{S}tothers triples}, Comment.
  Math. Univ. St. Pauli \textbf{54} (2005), no.~1, 49--68. \MR{2153955}

\bibitem{VHM-thesis}
Jeremy Van Horn-Morris, \emph{Constructions of open book decompositions},
  ProQuest LLC, Ann Arbor, MI, 2007, Thesis (Ph.D.)--The University of Texas at
  Austin. \MR{2710779}

\bibitem{Zhang-curve}
Weiyi Zhang, \emph{The curve cone of almost complex 4-manifolds}, Proc. Lond.
  Math. Soc. (3) \textbf{115} (2017), no.~6, 1227--1275. \MR{3741851}

\end{thebibliography}

\end{document}